\documentclass[a4paper, 11pt, headings = small, abstract]{scrartcl}
\usepackage[english]{babel}

\usepackage{amscd,amsmath,amsthm,array,hhline, enumerate, booktabs,fancybox,calc,textcomp,xcolor,graphicx,xspace,nicefrac,stmaryrd,url,arcs,listings,pifont,thmtools, setspace, amssymb}

\usepackage[semibold]{libertine}
\usepackage{mathtools}
\usepackage[upint,amsthm]{libertinust1math}
\usepackage[scr=boondoxupr, scrscaled = 1.0, cal =stixtwoplain]{mathalpha}
\usepackage[scaled=0.95]{inconsolata}
\usepackage{accents}

\DeclareMathAlphabet{\mathfrak}{U}{jkpmia}{m}{it}
\SetMathAlphabet{\mathfrak}{bold}{U}{jkpmia}{bx}{it}
\usepackage[utf8]{inputenc}

\usepackage{enumitem}
\usepackage{etoolbox}
\usepackage[normalem]{ulem}
\usepackage{geometry}
\usepackage{float}

\usepackage{bm}
\usepackage{tikz}
\usepackage[outdir =./]{epstopdf}
\usepackage[citestyle = numeric-comp,bibstyle = numeric, giveninits=true, natbib = true, backend = bibtex,maxbibnames=99, url=false]{biblatex}
\numberwithin{equation}{section}

\usepackage{chngcntr}
\counterwithin{figure}{section}

\usepackage{xcolor}
\usepackage[colorlinks=true, allcolors = myteal]{hyperref}

\usepackage{verbatim}

\definecolor{WIMgreen}{RGB}{60 134 132}
\definecolor{UMblue}{RGB}{4 47 86}
\definecolor{myteal}{RGB}{0 123 137}
\definecolor{material_green}{RGB}{27 43 52}
\definecolor{dracula_pink}{RGB}{180 93 149}
\definecolor{dracula_blue}{RGB}{40 42 54}
\definecolor{dracula_turq}{RGB}{92 143 159}
\definecolor{dracula_orange}{RGB}{255 184 108}
\definecolor{material_petrol}{RGB}{2 119 189}
\definecolor{Purple}{RGB}{103 58 183}
\definecolor{cs}{rgb}{0.0, 0.44, 1.0}
\definecolor{oucrimsonred}{rgb}{0.6, 0.0, 0.0}

\theoremstyle{plain}
\newtheorem{theorem}{Theorem}[section]
\newtheorem*{theorem*}{Theorem}
\newtheorem{proposition}[theorem]{Proposition}
\newtheorem{lemma}[theorem]{Lemma}

\theoremstyle{definition}
\newtheorem{definition}[theorem]{Definition}
\theoremstyle{remark}
\newtheorem{remark}[theorem]{Remark}
\newtheorem{example}[theorem]{Example}
\newtheorem{excont}{Example}

\setkomafont{sectioning}{\normalcolor\bfseries}
\setkomafont{descriptionlabel}{\normalcolor\bfseries}
\setkomafont{section}{\large}
\setkomafont{subsection}{\normalsize}
\setkomafont{subsubsection}{\normalsize}
\setkomafont{paragraph}{\normalsize}
\setkomafont{subparagraph}{\normalsize}
\setkomafont{author}{\large}
\setkomafont{title}{\large\bfseries}
\setkomafont{date}{\normalsize}

\def\supp{\operatorname{supp}}

\def\E{\mathbb{E}}

\def\N{\mathbb{N}}

\def\N{\mathbb{N}}

\def\R{\mathbb{R}}
\definecolor{darkred}{rgb}{0,0.6,0}

\def\cX{\mathcal{X}}

\newcommand{\cB}{\mathcal{B}}

\newcommand{\cF}{\mathcal{F}}

\newcommand{\PP}{\mathbb{P}}
\newcommand{\QQ}{\mathbb{Q}}

\renewcommand{\hat}{\widehat}

\renewcommand{\tilde}{\widetilde}%

\newcommand{\lebesgue}{\bm{\lambda}}
\newcommand{\overbar}[1]{\mkern 1.5mu\overline{\mkern-1.5mu#1\mkern-1.5mu}\mkern 1.5mu}

\DeclareMathOperator*{\esssup}{ess\,sup}

\restylefloat{table}

\newcommand*\diff{\mathop{}\!\mathrm{d}}

\newcommand{\one}{\mathbf{1}}

\newcommand{\vertiii}[1]{{\left\vert\kern-0.25ex\left\vert\kern-0.25ex\left\vert #1
\right\vert\kern-0.25ex\right\vert\kern-0.25ex\right\vert}}

\let\originalleft\left
\let\originalright\right
\renewcommand{\left}{\mathopen{}\mathclose\bgroup\originalleft}
\renewcommand{\right}{\aftergroup\egroup\originalright}

\newlist{todolist}{itemize}{2}
\setlist[todolist]{label=$\square$}

\makeatletter
\newcommand{\specificthanks}[1]{\@fnsymbol{#1}}
\makeatother

\bibliography{levy_kontrolle}

\allowdisplaybreaks

\makeatother
\title{\fontsize{14}{16} \selectfont Covariate shift in nonparametric regression with Markovian design}
\author{Lukas Trottner\thanks{Aarhus University, Department of Mathematics, Ny Munkegade 118, 8000 Aarhus C, Denmark \\
Email: \href{mailto:trottner@math.au.dk}{trottner@math.au.dk}}}
\date{\vspace{-5ex}}

\begin{document}
\renewcommand\thmcontinues[1]{continued}
\maketitle

\begin{abstract} 
Covariate shift in regression problems and the associated distribution mismatch between training and test data is a commonly encountered phenomenon in machine learning. In this paper, we extend recent results on nonparametric convergence rates for i.i.d.\ data to Markovian dependence structures. We demonstrate that under Hölder smoothness assumptions on the regression function, convergence rates for the generalization risk of a Nadaraya--Watson kernel estimator are determined by the similarity between the invariant distributions associated to source and target Markov chains. The similarity is explicitly captured in terms of a bandwidth-dependent similarity measure recently introduced in Pathak, Ma and Wainwright [\textit{ICML}, 2022]. Precise convergence rates are derived for the particular cases of finite Markov chains and  spectral gap Markov chains for which the similarity measure between their invariant distributions grows polynomially with decreasing bandwidth. For the latter,  we extend the notion of a distribution transfer exponent from Kpotufe and Martinet [\textit{Ann.\ Stat.}, 49(6), 2021] to kernel transfer exponents of uniformly ergodic Markov chains in order to generate a rich class of Markov kernel pairs for which convergence guarantees for the covariate shift problem can be formulated.
\end{abstract}
\section{Introduction} \label{sec:intro}
In the classical nonparametric regression setup the statistician is given i.i.d.\ covariate-response pairs $(X_i,Y_i)_{i = 1,\ldots,n}$ obeying the relation $Y_i = f^\ast(X_i) + \xi_i$ for some regression function $f^\ast$ and i.i.d.\ centered noise variables $(\xi_i)_{i = 1, \ldots, n}$. The task is to estimate $f^\ast$ and provide convergence rates w.r.t.\ the risk
\[\E\big[\lvert f^\ast(X) - \hat{f}(X) \rvert^2\big] = \int \E\big[\lvert f^\ast(x) - \hat{f}(x)\rvert^2\big] \,\PP_X(\diff{x}),\] 
for $X \sim X_1$ independent of $(X_i,\xi_i)_{i=1,\ldots,n}$, which can be interpreted as the prediction error or generalization risk for test data that is independent of the  training set $(X_i,Y_i)_{i=1,\ldots,n}$. However, for many applications it is not appropriate to assume that training and target distributions of the covariates are identical---that is the law of $X$ differs from the law of the training covariates $X_i$---begging the question how such \textit{covariate shift} influences the generalization performance of estimation methods that are known to provide optimal convergence rates in the classical setup. Such covariate shift problems along with the more general concept of  \textit{transfer learning} have drawn a lot of attention in the recent past \cite{bendavid10,cai21, sugi07, kpotufe21,christiansen22,pathak22,sh22} due to their natural occurence in applications in machine learning and other disciplines such as natural language processing \cite{yamada10,hassan13}, text image recognition \cite{zhang19}, computer vision \cite{larochelle20, wang18}, bioinformatics \cite{shaw18, mourragui19} or  economics \cite{heckman79}. 

The main inspiration for this paper is \citep{pathak22}, where a fundamental generalization error decomposition for the following i.i.d.\ covariate shift setup is proved: assume that a split sample of covariate response pairs $\mathcal{S} = (X_i,Y_i)_{i=1,\ldots,n}$ is observed  s.t.\ $\mathcal{S}_P \coloneq (X_1,\ldots, X_{n_P}) \overset{\mathrm{i.i.d.}}{\sim} P$ and $\mathcal{S}_Q \coloneq (X_{n_P+1},\ldots,X_{n_Q + n_P}) \overset{\mathrm{i.i.d.}}{\sim} Q$ are independent samples with sizes $n_P$ and $n_Q = n - n_P$, resp., and assume that the response mechanism $f^\ast(x) = \E[Y_i \mid X_i = x]$ is independent from the covariate distributions $P$ and $Q$. In other words, the training set of the model only partially represents the target distribution $Q$ with a possibly much larger batch of the training data coming from a distribution $P$ that has no prescribed relation to $Q$ other than sharing the same metric state space $\cX$. The regression function $f^\ast$, however,  is not influenced by the different distributional characteristics of source and target data. For the $Q$-generalization risk 
\[\int_{\cX} \E\big[\lvert f^\ast(x) - \hat{f}_h(x) \rvert^2\big] \, Q(\diff{x}),\] 
of a Nadaraya--Watson kernel estimator $\hat{f}_h$ with bandwith $h$ based on the sample $\mathcal{S}$, under Hölder smoothness assumptions on the regression function $f^\ast$, \cite{pathak22}  derive a decomposition in terms of a smoothness dependent bias and the dissimilarity between the weighted training distribution $\mu_n \coloneq (n_P P + n_Q Q)/n$ and target distribution $Q$ representing the difficulty of transfer from source to target distribution. This dissimilarity is expressed as $\rho_h(\mu_n,Q)$ for the bandwidth dependent (dis)similiarity measure 
\[\rho_h(\PP,\QQ) \coloneq \int_{\cX} \frac{1}{\PP(B(x,h))} \, \QQ(\diff{x}),\]
of probability distribution pairs $(\PP,\QQ)$ on $\cX$, where $B(x,h)$ is the closed ball in $\cX$ centered at $x$ with radius $h$. This error decomposition is  demonstrated to yield explicit minimax optimal estimation rates for source-target pairs $(P,Q)$ on $\cX = [0,1]$ belonging to a so called $\alpha$-family which is closely related to probability distribution pairs having a \textit{transfer exponent} as introduced in \cite{kpotufe21}.

\paragraph{Contribution and organization of the paper} 
Common to the above theoretical transfer learning studies is the assumption that the split training samples $\mathcal{S}_P$ and $\mathcal{S}_Q$ are both i.i.d. This assumption is clearly violated in a variety of applications, most strikingly  in situations where the training data is sequentially generated and thus a certain degree of temporal dependence is inherently present. To give a concrete example, this is the case in speech and language processing, where $n$-gram models and hidden Markov models are popular choices to approximate the sequential dependency of words and letters \cite{jurafsky09, manning99, plotz09}. The covariate shift issue is an especially prominent phenomenon for such applications due to the limited availability and expensiveness of converted documents or audio files into labeled training sets for a specific regression task. This may lead to little or no training data being available for the target data that may differ in several characteristics from the source data such as changes in environment, 
document type 
or genre.
This naturally raises the question of which  convergence guarantees can be given for the covariate shift problem when a specific dependence structure is underlying the training samples and how these guarantees compare to the situation when the data is i.i.d.

Regression problems with Markovian design have attracted increasing interest lately cf.\ \cite{fan22, gong22}. In this paper, we continue this line of research by providing a non-asymptotic convergence analysis for covariate shift regression given Markovian dependence  within the covariate samples $\mathcal{S}_P$ and $\mathcal{S}_Q$, where $P$ and $Q$ do no longer represent static probability distributions but the Markov kernels $P$ and $Q$ governing the temporal evolution of the Markovian data $\mathcal{S}_P = (X_0^{P}, \ldots, X^P_{n_P-1})$ and $\mathcal{S}_Q = (X_0^{Q}, \ldots, X^Q_{n_Q-1})$, respectively. The model is introduced in detail in Section \ref{sec:markovshift}.    We demonstrate that when  $\mathcal{S}_P$ and $\mathcal{S}_Q$ consist of (not necessarily stationary) Markov sequences having a \textit{pseudo spectral gap} and we consider the generalization risk 
\[\int_{\cX} \E\big[\lvert f^\ast(x) - \hat{f}_h(x) \rvert^2\big] \, \pi^Q(\diff{x}),\] 
where $\pi^Q$ is the invariant distribution associated to the target kernel $Q$, we have a similar decomposition of the risk for Hölder continuous regression functions as in \cite{pathak22}, with the dissimilarity $\rho_h(\mu_n,\pi^Q)$ between the ergodic mean training distribution $\mu_n \coloneq (n_P \pi^P + n_Q \pi^Q)/n$ and the target mean distribution $\pi^Q$ representing the difficulty of transfer from the source Markov kernel $P$ to the target Markov kernel $Q$. We then relate this bound to the prediction error 
\[\E\big[\big(f^\ast(X^Q_{n_Q + m}) - \hat{f}_h(X^Q_{n_Q+m})\big)^2\big],\] 
where the test data is no longer independent of the target training data $\mathcal{S}_Q$ but stems from the same Markov chain $X^Q$. 
The exact statements are given by Theorem \ref{theo:upper} and Theorem \ref{theo:pred}. The proof of Theorem \ref{theo:upper} fundamentally relies on concentration bounds for pseudo spectral gap Markov chains from \cite{paulin15}. These together with the underlying basics on spectral gap Markov chains are presented in Section \ref{sec:specgap}.

The general upper bound from Theorem \ref{theo:upper} is translated in Section \ref{sec:explicit} into explicit estimation rates for finite Markov chains as well as Markov chains whose invariant distributions belong to an $\alpha$-family, which we define as the natural multivariate extension of the class of distribution pairs on $\cX= [0,1]$ introduced in \cite{pathak22}. To construct explicit examples of Markov chains s.t.\ their invariant distributions satisfy the $\alpha$-family property, we develop an extension of the transfer exponent between probability distributions from \cite{kpotufe21} to Markov kernels of uniformly ergodic Markov chains. The relation of such chains having a transfer exponent and $\alpha$-families is discussed in detail and specific examples are given. Here we put a special  emphasis on the influence of the dimension of the target Markov chain $X^Q$, which as a convenient byproduct also extends the scalar i.i.d.\ discussion from \cite{pathak22} to higher dimensions. In particular, we demonstrate that in case of a dimension reduction from source to target data, the generalization rate can be faster than in the situation without covariate shift. 

\paragraph{Notation}
For two numbers $a,b \in \R$ we write $a \vee b \coloneq \max\{a,b\}$ and $a \wedge b \coloneq \min\{a,b\}$.  Given some metric space $(\cX,d)$,  $\cB(\cX)$ denotes the Borel $\sigma$-algebra on $\cX$. If $\cX = \R^{\mathsf{d}}$, $\mathsf{d} \in \N$, the metric $d$ is always understood to be induced by some vector norm. For a measure $\mu$ on $(\cX,\cB(\cX))$ and a function $h \colon \cX \to \R$ that is either nonnegative or such that $\int_{\cX} \lvert h(x) \rvert\, \mu(\diff{x}) < \infty$, we write $\mu(h) = \int_{\cX} h(x) \, \mu(\diff{x})$. By $L^p(\mu)$ we denote the usual $L^p$-space w.r.t.\ $\mu$ for $p \in [1,\infty]$ equipped with the norm $\lVert f \rVert_{L^p(\mu)} \coloneq \mu(\lvert f \rvert^p)^{1/p}$ for $p < \infty$ and $\lVert f\rVert_{L^\infty(\mu)} \coloneq \mu\text{-}\esssup \lvert f \rvert$ for $p = \infty$. The Lebesgue measure on $\R^{\mathsf{d}}$ is denoted by $\lebesgue$ or sometimes $\lebesgue_{\mathsf{d}}$ if we want to emphasize the dimension $\mathsf{d}$.  We let $B_d(x,r) \coloneq \{y \in \cX : d(x,y) \leq r\}$ be the closed $d$-balls centered at $x \in \cX$ with radius $r$ and denote by $N_{\cX}(\varepsilon,d)$  the $\varepsilon$-covering number of $\cX$ w.r.t.\ the metric $d$, i.e., the smallest number of $d$-balls with radius $\varepsilon$ s.t.\ their union is equal to $\cX$. If $\cX \subset \cX^\prime$ for some metric space $(\cX^\prime,d)$, we let $N^{\mathrm{ext}}_{\cX}(\varepsilon,d\vert_{\cX})$ be the external covering number of $\cX$, i.e., the smallest number of $d$-balls in $\cX^\prime$ with radius $\varepsilon$ s.t.\ their union contains $\cX$.   If there is no room for confusion we drop the indices  and simply write $B(x,r)$ for closed $d$-balls and $N(\varepsilon,d)$ for the $d$-covering number instead. If $\cX^\prime$ is compactly contained in $\cX$, we write $\cX^\prime \Subset \cX$. Finally, we let $[n] \coloneq \{1,\ldots,n\}$ for $n \in \N$ and for $B \subset \cX$ we let $\one_B$ be the indicator function w.r.t.\ $B$, i.e., $\one_B(x) = 1$ if $x \in B$ and $\one_B(x) = 0$ otherwise.

\section{Covariate shift regression model}\label{sec:markovshift}
We consider the following model under covariate shift. Let $(\cX,d)$ be a complete, separable metric space. Let $X^P= (X^P_n)_{n \in \N_0}$ be an irreducible and aperiodic  Markov chain on $\cX$ with transition operator $P$, unique invariant distribution $\pi^P$ and \textit{pseudo spectral gap} $\gamma^P_{\mathrm{ps}} > 0$, and $X^Q= (X^Q_n)_{n \in \N_0}$ be an irreducible and aperiodic Markov chain on $\cX$ that is independent of $X^P$ with transition operator $Q$, unique invariant distribution $\pi^Q$ and pseudo spectral gap $\gamma^Q_{\mathrm{ps}} > 0$. These notions are made precise in Section \ref{sec:specgap}. We do not require the chains to be stationary, but only to have  a \textit{warm start}. That is, we assume that $\PP_{X_0^P} = \mu^P$ for some initial distribution $\mu^P \ll \pi^P$, whose Radon--Nikodym derivative satisfies $\diff{\mu^P}/\diff{\pi^P} \in L^{\mathfrak{p}}(\pi^P)$ for some $\mathfrak{p} \in (1,\infty]$ and similarly assume $\PP_{X_0^Q} = \mu^Q \ll \pi^Q$ with $\diff{\mu^Q}/\diff{\pi^Q} \in L^{\mathfrak{q}}(\pi^Q)$ for some $\mathfrak{q} \in (1,\infty]$. For given $n, n_P, n_Q \in \N, n = n_P + n_Q$, we observe the covariate-response pairs $(X_i,Y_i)_{i = 1,\ldots, n}$, where 
\[X_i = \begin{cases} X_{i-1}^P, &\text{if } i \in \{1,\ldots,n_P\}\\ X_{i-1}^Q, &\text{if } i \in \{n_P +1, \ldots, n_P + n_Q\}, \end{cases} \]
and
\[Y_i = f^\ast(X_i) + \xi_i, \quad i =1, \ldots, n,\]
for some unknown regression function $f^\ast$ and a noise vector $(\xi_i)_{i \in [n]}$  given by
\[\xi_i = \begin{cases} \xi_{i-1}^P, &\text{if } i \in \{1,\ldots,n_P\}\\ \xi_{i-1}^Q, &\text{if } i \in \{n_P +1, \ldots, n_P + n_Q\}, \end{cases} \]
where the heteroskedastic random sequences $\xi^P = (\xi_i^P)_{i \in \N_0} \subset L^2(\PP)$ and $\xi^Q= (\xi_i^Q)_{i \in \N_0} \subset L^2(\PP)$ satisfy the following assumptions: 
\begin{enumerate} 
\item $(X^P,\xi^P)$ and $(X^Q,\xi^Q)$ are independent;
\item for $R \in \{P,Q\}$, the elements of $\xi^R$ are conditionally independent given $X^R$;
\item for $R \in \{P,Q\}$ and some $\sigma > 0$ it holds  $\PP$-a.s.for any $i \in \N_0$,
\[\E[\xi_i^R \mid X^R]  = 0, \quad \E[(\xi_i^R)^2 \mid X^R] \leq \sigma^2.\] 
\end{enumerate}
These assumptions are  satisfied, e.g., when $\xi^R_i = \varsigma(X_i^R) \varepsilon_i^R$ for some i.i.d.\ sequences $(\varepsilon^R_i)_{i \in N_0}$ independent of each other and of $(X^P,X^Q)$ with $\E[\varepsilon_1^R] = 0$, $\E[(\varepsilon_1^R)^2] \leq 1$, and $\varsigma \colon \cX \to [-\sigma,\sigma]$. 
By $\mu_n$ we define the ergodic mean training distribution induced by the sample sizes $n_P$ and $n_Q$, given as  
\[\mu_n \coloneq \frac{n_P \pi^P + n_Q \pi^Q}{n}.\]
In the following we will refer to this setup as the \textit{$(P,Q)$-Markov covariate shift regression model}. The covariate shift regression model from \cite{pathak22} for i.i.d.\  vectors $X^P$ and $X^Q$ is included as a special case in the Markov covariate shift model for we may always interpret $X^P$ and $X^Q$ as Markov chains with transition operators $Ph \coloneq \pi^P(h), Qh \coloneq \pi^Q(h)$ having both \textit{absolute} as well as pseudo spectral gap $1$.  

We are interested in the $L^2$-generalization risk w.r.t.\ the target invariant distribution $\pi^Q$,
\[\E\big[\lVert \hat{f}_h - f^\ast \rVert^2_{L^2(\pi^Q)} \big]\] 
for a Nadaraya--Watson type kernel estimator $\hat{f}_h$ with bandwidth $h > 0$  that we introduce in Section \ref{sec:nw}. Thus, we test the estimator $\hat{f}_h$ that is based on training observations $(\mathcal{S}_P, (Y_i)_{i \in [n_P]})$ and $(\mathcal{S}_Q, (Y_i)_{i= n_P+1,\ldots,n})$ for independent Markov samples $\mathcal{S}_P = \{X^P_0,\ldots, X^P_{n^P-1}\}$ and $\mathcal{S}_Q = \{X^Q_0,\ldots, X^Q_{n^P-1}\}$ against an independent data sample coming from some stationary $Q$-Markov chain. 
Similarly to \cite{pathak22}, we show that the generalization risk may be decomposed into a smoothness dependent bias and a stochastic error expressed  in terms of the (dis)similarity measure $\rho_h(\mu_n,\pi^Q)$ between the mixture training distribution $\mu_n$ and the target distribution $\pi^Q$, where we recall from Section \ref{sec:intro} that  for two probability measures $\PP,\mathbb{Q}$ on $(\mathcal{X},d)$,
\[\rho_h(\PP,\mathbb Q) \coloneq \int_{\cX} \frac{1}{\PP(B(x,h))} \, \mathbb{Q}(\diff{x}), \quad h > 0,\] 
where $\tfrac{1}{0} \coloneq +\infty$. We also discuss implications for the prediction error 
\[\E\big[\big(\hat{f}_h(X^Q_{n_Q+m}) - f^\ast(X^Q_{n_Q + m})\big)^2\big],\] 
where the test sample is not independent of the training data but originates from the same Markov chain underlying $\mathcal{S}_Q$.

Intuitively, the similarity measure is small when $\QQ$ has no significant mass in low probability regions of $\PP$ and becomes large otherwise. In this regard, examples of probability pairs  $(\PP,\QQ)$ for the extremal case $\rho_h(\PP,\QQ) = + \infty$ may be easily constructed when $\QQ$ is not absolutely continuous w.r.t.\ $\PP$, see Lemma \ref{lem:explosion}. However, there are also examples when $\PP$ and $\QQ$ are singular but $\rho_h(\PP,\QQ) < \rho_h(\PP,\PP)$. We explore this phenomenon in more detail in Section \ref{sec:explicit} and Appendix \ref{app:exalpha}, where for $\cX \subset \R^{\mathsf{d}}$ it is shown that  if the support of $\pi^Q$ is contained in some lower dimensional subspace of $\cX$, $\rho_h(\pi^P, \pi^Q) < \rho_h(\pi^P,\pi^P)$ is possible for appropriate Markov chains $X^P$ and $X^Q$. 

\section{Markov chains with pseudo spectral gap}\label{sec:specgap}
Extending the results from \cite{pathak22} to the Markovian setting requires a precise understanding of the concentration properties of the chains, which have been subject to intensive research efforts in recent decades. To this end, we focus on models belonging to the quite general and not necessarily symmetric class of Markov chains with pseudo spectral gap introduced in \cite{paulin15}. More precisely we will use the  Bernstein inequality obtained in \cite{paulin15}, which complements Hoeffding type inequalities from \cite{miasojedow14,fan21,perron04} in the spectral gap setting that we now introduce. For a more detailed background we refer to \cite[Chapter 22]{douc18}.

Let $(\cX,d)$ be a complete, separable metric space and let $Z = (Z_n)_{n \in \N_0}$ be an $\mathcal{X}$-valued $\varphi$-irreducible and aperiodic Markov chain on some probability space $(\Omega,\cF,\PP)$ with transition operator $P$ and unique invariant distribution $\pi$. Thus, for $\PP^x \coloneq \PP(\cdot \mid X_0 = x)$ and $\PP^\mu(\cdot) \coloneq \int_{\cX} \PP^x(\cdot)\, \mu(\diff{x})$ for a probability measure $\mu$ on $(\cX,\cB(\cX))$ we have $Pf(x) = \E^x[f(Z_1)]$ and the stationarity property $\E^\pi[f(Z_n)] = \pi(f)$ given $f\colon \cX \to \mathbb{C}$ for which the integrals are well-defined. By the usual slight abuse of notation, we will not distinguish between the transition operator $P$ and the Markov kernel $P\colon \cX \times \mathcal{B}(\cX) \to [0,1]$ given by $P(x,B) = \PP^x(Z_1 \in B)$ for $(x,B) \in \cX \times \cB(\cX)$. In particular, we have for $f$ as above,
\[P^nf(x) = \E^x[f(Z_n)] = \int_{\cX^n} P(x,\diff{x_1}) \prod_{i=1}^{n-1} P(x_i,\diff{x_{i+1}})\, f(x_n),\]
and we let $P^n(x,B) \coloneq \PP^x(Z_n \in B) = P^n \one_B(x)$ for $(x,B) \in \cX \times \cB(\cX)$. 

Let $L^2(\pi)$ be the Hilbert space of square integrable functions w.r.t.\ $\pi$ endowed with the inner product  $\langle g,h \rangle_\pi \coloneq \pi(g \overbar{h})$ and norm $\lVert h\rVert_\pi \coloneq \pi(\lvert h \rvert^2)^{1/2} = \lvert \langle h,  h\rangle_{\pi}\rvert^{1/2}$.
$P$ can be extended to a bounded linear operator on $L^2(\pi)$ and we denote by $P^\ast$ its adjoint. For any bounded linear operator $A$ on $L^2(\pi)$ let $\vertiii{A}_{\pi} \coloneq \sup_{\lVert f \rVert_\pi = 1} \lVert Af \rVert_{\pi}$ be its operator norm and $\rho(A)$ be its spectral radius. We say that $P$ has \textit{absolute spectral gap} $\gamma^P_\ast$ if 
\[\gamma^P_\ast = 1 - \varrho(P - \Pi) > 0,\]
where $\Pi(h) \coloneq \pi(h)$ for $h \in L^2(\pi)$. Taking into account Gelfand's formula 
\[\varrho(P - \Pi) = \lim_{n  \to \infty} \vertiii{(P - \Pi)^n}^{1/n}_{\pi} = \lim_{n \to \infty} \vertiii{P^n - \Pi}^{1/n}_\pi,\] 
where we used invariance of $\pi$ for the second equality, it is clear that the existence of an absolute spectral gap is equivalent to $L^2(\pi)$-geometric convergence in the sense $\vertiii{P^n - \Pi}_\pi \leq c\delta^n$ for some $\delta \in (0,1)$.
If additionally $Z$ is reversible, i.e., $P$ is self-adjoint, we have precisely $\delta = \varrho(P - \Pi) \in (0,1)$, since normality of $P - \Pi$ gives
\[\vertiii{P^n - \Pi}_{\pi} = \vertiii{(P -\Pi)^n}_{\pi} =\vertiii{P - \Pi}^n_{\pi} = \varrho(P - \Pi)^n.\]
With a view towards concentration inequalities for non-reversible chains, \cite{paulin15} introduces the notion of a \textit{pseudo spectral gap} related to multiplicative reversibilizations of $P$. We say that $Z$ has a \textit{pseudo spectral gap} if 
\[\gamma_{\mathrm{ps}}^P \coloneq \max_{k \in \N} \Big\{\frac{1 - \varrho((P^\ast)^k P^k - \Pi)}{k}\Big\} > 0.\]
which exists as soon as for some integer $k$ the positive, self-adjoint Markov operator $(P^\ast)^k P^k$ has an absolute spectral gap. Note in particular that if $Z$ is reversible with absolute spectral gap, we have
\begin{equation}\label{eq:abs_pseudo}
\gamma^P_{\mathrm{ps}} \geq 1 - \vertiii{(P - \Pi)^2}_{\pi} = 1 - \varrho(P - \Pi)^2 \geq 1 - \varrho(P - \Pi) = \gamma_{\ast}^P > 0.
\end{equation} 
Consequently, a spectral gap of a reversible chain always provides a lower bound for the pseudo spectral gap.
A general class of Markov chains with pseudo spectral gap studied in \cite{paulin15} are \textit{uniformly} ergodic Markov chains, i.e., chains s.t.\ for some constants $c > 0, \kappa \in (0,1)$,
\begin{equation}\label{eq:uni_erg}
\lVert P^n(x,\cdot) - \pi \rVert_{\mathrm{TV}} \leq c\kappa^n, \quad x \in \cX, n \in \N,
\end{equation} 
where $\lVert \cdot \rVert_{\mathrm{TV}}$ denotes the total variation norm of a finite signed measure. An application of the Riesz--Thorin interpolation theorem shows that 
\[\vertiii{P^n - \Pi}_{\pi} \leq 2 \sup_{x \in \cX} \lVert P^n(x,\cdot) - \pi \rVert_{\mathrm{TV}}^{1/2} \leq 2c \kappa^{n/2}, \]
cf., e.g., \cite[Section 2.1]{ferre12}, whence uniformly ergodic Markov chains admit an absolute spectral gap $\gamma_\ast^P \geq 1 -\sqrt{\kappa}$.
Moreover, uniformly ergodic chains always have a pseudo spectral gap and if we denote the \textit{mixing time} $\tau^P$ as 
\[\tau^P \coloneq \min\big\{n \in \N: \sup_{x \in \cX}\lVert P^n(x,\cdot) - \pi \rVert_{\mathrm{TV}} \leq 1/4\big\},\] 
then according to \cite[Proposition 3.4]{paulin15}, the pseudo spectral gap is controlled by the mixing time via 
\[\gamma_{\mathrm{ps}}^P \geq \frac{1}{2\tau^P}.\]
To make this lower bound explicit for specific kernels $P$, it is useful to note that uniform geometric ergodicity \eqref{eq:uni_erg} is equivalent to a variant of the classical Doeblin recurrence condition (cf.\ \cite[Theorem 16.0.2]{mt09}), that is, we have the minorization property
\begin{equation} \label{eq:doeblin}
P^m(x,\cdot) \geq \varepsilon \nu, \quad x \in \mathcal{X},
\end{equation}
for some $m \in \N$, where $\varepsilon \in (0,1]$ and $\nu$ is some probability measure on $(\cX,\mathcal{B}(\cX))$. In other words, $Z$ is uniformly ergodic if and only if the state space $\cX$ is \textit{small}, cf.\ \cite[Chapter 5]{mt09}. Given \eqref{eq:doeblin} the uniform rate of convergence in \eqref{eq:uni_erg} becomes explicit with 
\begin{equation}\label{eq:doeblin_rate}
\kappa = (1-\varepsilon)^{1/m}, \quad c = 2(1-\varepsilon)^{-1}, 
\end{equation}
cf.\ \cite[Theorem 16.2.4]{mt09}, yielding also an explicit expression for $\tau^P$ and a corresponding lower bound for $\gamma_{\mathrm{ps}}^P$. 
Such minorization properties for the semigroup are usually not satisfied when $\mathcal{X}$ is unbounded, in which case one can only hope for at most geometric ergodicity, that is, 
\[\lVert P^n(x,\cdot) - \pi^P \rVert_{\mathrm{TV}} \leq V(x)\kappa^n, \quad x \in \cX, n \in \N,\]
for some \textit{unbounded} penalty function $V$ and $\kappa \in (0,1)$. It is important to note here that if $Z$ is reversible, then geometric ergodicity  for $\pi$-a.e.\ initial value $x \in \cX$ is actually equivalent to the existence of an absolute spectral gap \cite[Theorem 2]{roberts01}. However, in the non-reversible case, by \cite[Theorem 1.4]{konto12} there exist geometrically ergodic chains that do not admit an absolute spectral gap.

Let us now introduce the concentration results from \cite{paulin15}. Let $\gamma_{\mathrm{ps}} \equiv \gamma^P_{\mathrm{ps}} > 0$ and $f\colon \cX \to \R$ be bounded. Then, by \cite[Theorem 3.4]{paulin15} we have the Bernstein inequality
\[\PP^\pi\Big(\sum_{i=0}^{n-1} (f(Z_i) - \pi(f)) \geq x \Big) \leq \exp\Big(- \frac{\gamma_{\mathrm{ps}}x^2}{8(n + \gamma_{\mathrm{ps}}^{-1})\sigma^2(f) + 20\lVert f - \pi(f) \rVert_\infty x}\Big), \quad x > 0,\]
where $\sigma^2(f) = \pi(f^2) - \pi(f)^2$.
If a measure $\mu$ on $(\cX,\cB(\cX))$ satisfies $\mu \ll \pi$ with $\diff \mu /\diff \pi \in L^p(\pi)$ for some $p \in (1,\infty]$, it now follows easily from Hölder's inequality that
\begin{equation}\label{eq:bernstein}
\PP^\mu\Big(\sum_{i=0}^{n-1} (f(Z_i) - \pi(f)) \geq x \Big) \leq \exp\Big(- \frac{1}{\overbar{p}}\frac{\gamma_{\mathrm{ps}}x^2}{8(n + \gamma_{\mathrm{ps}}^{-1})\sigma^2(f) + 20\lVert f - \pi(f) \rVert_\infty x}\Big), \quad x > 0,
\end{equation}
where for $p > 1$ we denote by 
\[\overbar{p} \coloneq \begin{cases} p/(p-1), &\text{if } p < \infty,\\ 1, &\text{if } p = \infty,\end{cases}\] 
its conjugate Hölder exponent.
This concentration inequality yields the following negative moment bound for additive functionals of $Z$, which will be instrumental in the proof of Theorem \ref{theo:upper}. 
\begin{lemma}\label{lem:negmom}
Suppose that $Z$ has pseudo spectral gap $\gamma_{\mathrm{ps}} > 0$. Let $\mu$ be a measure on $(\cX,\cB(\cX))$ s.t.\ $\mu \ll \pi$ with $\diff \mu /\diff{\pi} \in L^p(\mu)$ for some $p \in (1,\infty]$. For any bounded, nonnegative function $f$ such that $\pi(f) > 0$, it holds
\[\E^\mu\Big[ \frac{1}{1 + \sum_{i=0}^{n-1} f(Z_i)}\Big] \leq  \frac{4\big\lVert \frac{\diff{\mu}}{\diff \pi} \big\rVert_{L^p(\pi)}(20\overbar{p}\lVert f -\pi(f) \rVert_\infty \gamma_{\mathrm{ps}}^{-1} +1)}{n\pi(f)}.\]
\end{lemma}
\begin{proof} 
The bound is trivial when $f$ is constant, so let us suppose it is not, implying $\lVert f - \pi(f) \rVert_\infty > 0$. Let $\alpha \in (0,1)$ and set $S_n(g) \coloneq \sum_{i=0}^{n-1} g(Z_i)$. We have 
\begin{align*}
&\E^\mu\Big[ \frac{1}{1 + \sum_{i=0}^{n-1} f(Z_i)}\Big]\\
&\quad= \E^\mu\Big[ \frac{1}{1 + S_n(f)}\,;\, S_n(f) > (1-\alpha)n\pi(f)\Big] + \E^\mu\Big[ \frac{1}{1 + S_n(f)}\,;\, S_n(f) - n\pi(f) \leq -\alpha n\pi(f)\Big]\\
&\quad\leq \frac{1}{(1-\alpha) n\pi(f)} + \PP^\mu\big(S_n(f - \pi(f)) \leq - \alpha n\pi(f)\big)\\
&\quad\leq \Big\lVert \frac{\diff{\mu}}{\diff \pi} \Big \rVert_{L^p(\pi)}\Big(\frac{1}{(1-\alpha)n\pi(f)} + \exp\Big(- \frac{1}{\overbar{p}}\frac{\gamma_{\mathrm{ps}} (\alpha n\pi(f))^2}{16n\sigma^2(f) + 20n\alpha \lVert f -\pi(f) \rVert_\infty \pi(f) }\Big)\Big),
\end{align*}
where we used $S_n(f) \geq 0$ by $f \geq 0$ for the second line and \eqref{eq:bernstein} for the last line, taking into account $n \geq \gamma_{\mathrm{ps}}^{-1}$. Since $f \geq 0$, it holds 
\[\sigma^2(f) \leq \pi(f)\lVert f\rVert_\infty -\pi(f)^2 \leq \lVert f -\pi(f)\rVert_\infty \pi(f),\] 
whence 
\begin{align*} 
\exp\Big(- \frac{1}{\overbar{p}}\frac{\gamma_{\mathrm{ps}} (\alpha n\pi(f))^2}{16n\sigma^2(f) + 20n\alpha \lVert f -\pi(f)\rVert_\infty \pi(f) }\Big) &\leq \exp\Big(- \frac{1}{\overbar{p}}\frac{\gamma_{\mathrm{ps}} (\alpha n\pi(f))^2}{(16 +20\alpha) \lVert f -\pi(f)\rVert_\infty n\pi(f) }\Big) \\
&= \exp\Big(- \frac{1}{\overbar{p}}\frac{\gamma_{\mathrm{ps}} \alpha^2 n\pi(f)}{(16 +20\alpha) \lVert f -\pi(f) \rVert_\infty }\Big)\\ 
&\leq \frac{20\overbar{p}(1 + \alpha)\lVert f - \pi(f)\rVert_\infty}{\gamma_{\mathrm{ps}}\alpha^2 n\pi(f)}.
\end{align*}
Setting $c = 20 \overbar{p}\lVert f - \pi(f) \rVert_\infty/\gamma_{\mathrm{ps}}$, we have $(1 - \alpha)^{-1} = c(1+\alpha)\alpha^{-2}$ for $\alpha = \sqrt{c/(c+1)}$. Thus, for this choice of $\alpha$, we obtain from above 
\begin{align*}
\E^\mu\Big[ \frac{1}{1 + \sum_{i=0}^{n-1} f(Z_i)}\Big] \leq 2\Big\lVert \frac{\diff{\mu}}{\diff \pi} \Big \rVert_{L^p(\pi)} \frac{1}{(1-\sqrt{c/(c+1)})n\pi(f)} &= 2\Big\lVert \frac{\diff{\mu}}{\diff \pi} \Big \rVert_{L^p(\pi)} \frac{1+ \sqrt{c/(c+1)}}{(1- c/(c+1))n \pi(f)}\\
&\leq \frac{4 \big\lVert \frac{\diff{\mu}}{\diff \pi} \big\rVert_{L^p(\pi)} (c+1)}{n \pi(f)}\\
&= \frac{4\big\lVert \frac{\diff{\mu}}{\diff \pi} \big\rVert_{L^p(\pi)}(20\overbar{p}\lVert f -\pi(f) \rVert_\infty \gamma_{\mathrm{ps}}^{-1} +1)}{n\pi(f)}.
\end{align*}
\end{proof}

\section{Convergence analysis of the Nadaraya--Watson estimator under covariate shift} \label{sec:nw}
Let a $(P,Q)$-Markov covariate shift regression model be given. Denote $\mathcal{G}_n \coloneq \bigcup_{i=1}^n B(X_i,h_n)$ and consider the estimator 
\[\hat{f}_n(x) \coloneq \frac{\sum_{i=1}^n Y_i \one_{B(x,h_n)}(X_i)}{\sum_{i=1}^n \one_{B(x,h_n)}(X_i)} \one_{\mathcal{G}_n}(x), \quad x \in \mathcal{X},\]
for some $h_n > 0$, which is a Nadaraya--Watson (NW) type kernel estimator with uniform kernel $K = \one_{B(0,1)}$. We now prove an MSE upper bound for $\hat{f}_n$ under an Hölder continuity assumption on the regression function $f^\ast$ that provides a natural analogue to the case of i.i.d.\ covariates under covariate shift given in \cite[Theorem 1]{pathak22}. For $0< \beta \leq 1$ let 
\[\mathcal{H}(\beta,L) \coloneq \big\{f \colon \mathcal{X} \to \R: \lvert f(x) -f(y) \rvert \leq L \lvert d(x,y) \rvert^\beta \big\},\]
be the class of $(\beta,L)$-Hölder continuous functions on $\mathcal{X}$. Considering Hölder continuous regression functions  is natural by construction of the estimator with a kernel $K$ of order $1$, see \cite[Chapter 1]{tsy09}. We do not require $f^\ast$ to be bounded but only $f^\ast \in L^2(\pi^Q)$.

\begin{theorem} \label{theo:upper}
Suppose that $f^\ast \in \mathcal{H}(\beta,L) \cap L^2(\pi^Q)$. Then, for $n_P \geq (\gamma^P_{\mathrm{ps}})^{-1}, n_Q \geq (\gamma^Q_{\mathrm{ps}})^{-1}$, it holds 
\[\E\Big[\big\lVert \hat{f}_n - f^\ast \big\rVert_{L^2(\pi^Q)}^2 \Big] \leq  L^2h_n^{2\beta} + \mathfrak{C}  \frac{\pi^Q(\lvert f^\ast \rvert^2) + \sigma^2}{n} \rho_{h_n}(\mu_n,\pi^Q), \]
for 
\begin{align*}
\mathfrak{C} = \mathfrak{C}(\lambda^P_{\mathrm{ps}}, \lambda^Q_{\mathrm{ps}}, \mu^P,\mu^Q,\mathfrak{p},\mathfrak{q}) \coloneq  12 \max\Big\{&3 \Big\lVert\frac{\diff{\mu^P}}{\diff \pi^P}\Big\rVert_{L^{\mathfrak{p}}(\pi^P)}\Big\lVert\frac{\diff{\mu^Q}}{\diff \pi^Q}\Big\rVert_{L^{\mathfrak{q}}(\pi^Q)} \big(\overbar{\mathfrak{p}} (\gamma^P_{\mathrm{ps}})^{-1} \vee \overbar{\mathfrak{q}} (\gamma^Q_{\mathrm{ps}})^{-1}\big),\\
 & 28\Big(\Big\lVert\frac{\diff{\mu^P}}{\diff \pi^P}\Big\rVert_{L^{\mathfrak{p}}(\pi^P)}\overbar{\mathfrak{p}} (\gamma^P_{\mathrm{ps}})^{-1} \vee \Big\lVert\frac{\diff{\mu^Q}}{\diff \pi^Q}\Big\rVert_{L^{\mathfrak{q}}(\pi^Q)} \overbar{\mathfrak{q}}(\gamma^Q_{\mathrm{ps}})^{-1} \Big) \Big\}.
\end{align*}
\end{theorem}
\begin{proof} 
By our assumptions on the Hölder regularity of $f^\ast$ and the noise vector $(\xi_i)_{i=1,\ldots,n}$, an adapted statement of Lemma 2 from \cite{pathak22} remains valid, that is, for $\overbar{f}_n(x) \coloneq \E[\hat{f}_n(x) \mid X_1,\ldots, X_n]$ and $\mathcal{G}_n \coloneq \bigcup_{i=1}^n B(X_i,h)$, we have the bounds
\begin{align} 
(\overbar{f}_n(x) - f^\ast(x))^2 &\leq \lvert f^\ast(x)\rvert^2 \one_{\mathcal{G}_n^{\mathrm{c}}}(x) + L^2 h^{2\beta}_n, \label{eq:bias}\\ 
\E\big[(\overbar{f}_n(x) - \hat{f}_n(x))^2 \mid X_1,\ldots,X_n \big] &\leq \frac{\sigma^2 \one_{\mathcal{G}_n}(x)}{\sum_{i=1}^n \one_{B(x,h_n)}(X_i)}.\label{eq:stoch}
\end{align}
Indeed, since $\E[\xi_i \mid X_1,\ldots,X_n] = 0$, we have 
\[\overbar{f}_n(x) = \frac{\sum_{i=1}^n f^\ast(X_i) \one_{B(x,h_n)}(X_i)}{\sum_{i=1}^n \one_{B(x,h_n)}(X_i)} \one_{\mathcal{G}_n}(x), \quad x \in \mathcal{X},\]
whence the conditional bias bound \eqref{eq:bias} follows exactly as in \cite{pathak22}  using  Hölder regularity of $f^\ast$ for $x \in \mathcal{G}_n$. Moreover,  using the assumptions on the noise vector $\xi$ we a.s.\ have, for any $i,j \in [n]$,
\[\E[\xi_i\xi_j \mid X_1,\ldots,X_n] = \E[\xi_i \mid X_1,\ldots,X_n] \E[\xi_j \mid X_1,\ldots,X_n]\one_{\{i \neq j\}} +  \E[\xi_i^2 \mid X_1,\ldots,X_n] \one_{\{i=j\}} \leq \sigma^2 \one_{\{i=j\}},\] 
whence inequality \eqref{eq:stoch} follows from observing that 
\[\E\big[(\overbar{f}_n(x) - \hat{f}_n(x))^2 \mid X_1,\ldots,X_n \big] = \frac{\one_{\mathcal{G}_n}(x)}{(\sum_{i=1}^n \one_{B(x,h_n)}(X_i))^2}\sum_{i,j = 1}^n \E[\xi_i\xi_j \mid X_1\ldots, X_n] \one_{B(x,h_n)}(X_i)\one_{B(x,h_n)}(X_j).\]
Now, since
\[\E[(\hat{f}_n(x)-\overbar{f}_n(x))(\overbar{f}_n(x) -f^\ast(x))] = \sum_{i=1}^n\E\Big[\frac{\E[\xi_i \mid X_1,\ldots,X_n] \one_{B(x,h_n)}(X_i)}{\sum_{j=1}^n \one_{B(x,h_n)}(X_j)}(\overbar{f}_n(x) - f^\ast(x)) \one_{\mathcal{G}_n}(x) \Big] = 0,\] 
we obtain the bias-variance type decomposition 
\[\E\big[(\hat{f}_n(x) - f^\ast(x))^2\big] = \E\big[(\overbar{f}_n(x) - f^\ast(x))^2\big] + \E\big[\E\big[(\overbar{f}_n(x) - \hat{f}_n(x))^2 \mid X_1,\ldots,X_n \big]\big],\]
and hence with \eqref{eq:bias}, \eqref{eq:stoch}, 
\begin{equation} \label{eq:biasvar}
\E\big[(\hat{f}_n(x) - f^\ast(x))^2\big] \leq L^2 h_n^{2\beta} + \lvert f^\ast(x) \rvert^2 \E\big[\one_{\mathcal{G}_n^{\mathrm{c}}}(x)\big] + \sigma^2 \E\Big[\frac{\one_{\mathcal{G}_n}(x)}{\sum_{i=1}^n \one_{B(x,h_n)}(X_i)} \Big] \eqcolon L^2 h_n^{2\beta} + \mathcal{I}_1^x + \mathcal{I}_2^x.
\end{equation}
To bound $\mathcal{I}_1^x$, first note that by independence of the Markov chains $X^Q$ and $X^P$, we have 
\[\E[\one_{\mathcal{G}_n^{\mathrm{c}}}(x)] = \PP\Big(\sum_{i=1}^{n_P} \one_{B(x,h_n)}(X_{i-1}^P) = 0 \Big) \PP\Big(\sum_{i=1}^{n_Q} \one_{B(x,h_n)}(X_{i-1}^Q) = 0 \Big).\]
Let $x \in \mathcal{X}$ s.t.\ $\pi^P(B(x,h_n)) > 0$. Using the Bernstein inequality \eqref{eq:bernstein} under the assumption $n_P \geq (\gamma^P_{\mathrm{ps}})^{-1}$ and the fact that $\PP_{X_0^{P}} = \mu^P$, we obtain 
\begin{align*} 
&\PP\Big(\sum_{i=1}^{n_P} \one_{B(x,h_n)}(X_{i-1}^P) = 0 \Big) \\
&\,\leq \PP\Big(\sum_{i=0}^{n_P-1} \big(\one_{B(x,h_n)}(X_i^P) - \pi^P(B(x,h_n)) \leq -n_P\pi^P(B(x,h_n)) \Big) \\ 
&\,\leq \Big\lVert \frac{\diff{\mu^P}}{\diff \pi^P}\Big\rVert_{L^{\mathfrak{p}}(\pi^P)}\exp\Big(- \frac{\gamma^P_{\mathrm{ps}} \overbar{\mathfrak{p}}^{-1}(n_P\pi^P(B(x,h_n)))^2}{16 n_P\pi^P(B(x,h_n))(1-\pi^P(B(x,h_n))) + 20 n_P\pi^P(B(x,h_n))} \Big)\\ 
&\,\leq \Big\lVert \frac{\diff{\mu^P}}{\diff \pi^P}\Big\rVert_{L^{\mathfrak{p}}(\pi^P)}\exp\Big(- \frac{n_P \pi^P(B(x,h_n))}{36 \overbar{\mathfrak{p}} (\gamma^P_{\mathrm{ps}})^{-1}} \Big).
\end{align*}
Noting that this bound is trivially true if $\pi(B(x,h_n)) = 0$ we obtain for any $x \in \mathcal{X}$, 
\[\PP\Big(\sum_{i=1}^{n_P} \one_{B(x,h_n)}(X_{i-1}^P) = 0 \Big) \leq \Big\lVert \frac{\diff{\mu^P}}{\diff \pi^P}\Big\rVert_{L^{\mathfrak{p}}(\pi^P)}\exp\Big(- \frac{n_P \pi^P(B(x,h_n))}{36 \overbar{\mathfrak{p}} (\gamma^P_{\mathrm{ps}})^{-1}} \Big).\]
Similarly, we obtain for any $x \in \mathcal{X}$,
\[\PP\Big(\sum_{i=1}^{n_Q} \one_{B(x,h_n)}(X_{i-1}^Q) = 0 \Big)\leq \Big\lVert \frac{\diff{\mu^Q}}{\diff \pi^Q}\Big\rVert_{L^{\mathfrak{q}}(\pi^Q)}\exp\Big(- \frac{n_Q \pi^Q(B(x,h_n))}{36 \overbar{\mathfrak{q}} (\gamma^Q_{\mathrm{ps}})^{-1}} \Big),\]
and hence 
\begin{equation}\label{eq:i1}
\begin{split}
\mathcal{I}_1^x &\leq \lvert f^\ast(x) \rvert^2\Big\lVert\frac{\diff{\mu^P}}{\diff \pi^P}\Big\rVert_{L^{\mathfrak{p}}(\pi^P)}\Big\lVert\frac{\diff{\mu^Q}}{\diff \pi^Q}\Big\rVert_{L^{\mathfrak{q}}(\pi^Q)} \exp\Big(- \frac{n_P \pi^P(B(x,h_n)) + n_Q \pi^Q(B(x,h_n))}{36 (\overbar{\mathfrak{p}} (\gamma^P_{\mathrm{ps}})^{-1} \vee \overbar{\mathfrak{q}} (\gamma^Q_{\mathrm{ps}})^{-1})} \Big)\\
& \leq \lvert f^\ast(x) \rvert^2\Big\lVert\frac{\diff{\mu^P}}{\diff \pi^P}\Big\rVert_{L^{\mathfrak{p}}(\pi^P)}\Big\lVert\frac{\diff{\mu^Q}}{\diff \pi^Q}\Big\rVert_{L^{\mathfrak{q}}(\pi^Q)} \frac{36 (\overbar{\mathfrak{p}} (\gamma^P_{\mathrm{ps}})^{-1} \vee \overbar{\mathfrak{q}} (\gamma^Q_{\mathrm{ps}})^{-1})}{n \mu_n(B(x,h_n))},
\end{split}
\end{equation}
follows for any $x \in \mathcal{X}$. Moreover, using Lemma \ref{lem:negmom}, we have
\begin{equation}\label{eq:i2}
\begin{split}
\mathcal{I}_2^x &\leq 2\sigma^2 \bigg(\E\bigg[\frac{1}{1 + \sum_{i=1}^{n_P} \one_{B(x,h_n)}(X_{i-1}^P)} \bigg] \wedge \E\bigg[\frac{1}{1 + \sum_{i=1}^{n_Q} \one_{B(x,h_n)}(X_{i-1}^Q)} \bigg] \bigg)\\
&\leq \frac{168\sigma^2 \Big(\overbar{\mathfrak{p}}\big\lVert\frac{\diff{\mu^P}}{\diff \pi^P}\big\rVert_{L^{\mathfrak{p}}(\pi^P)} (\gamma^P_{\mathrm{ps}})^{-1} \vee \overbar{\mathfrak{q}}\big\lVert\frac{\diff{\mu^Q}}{\diff \pi^Q}\big\rVert_{L^{\mathfrak{q}}(\pi^Q)} (\gamma^Q_{\mathrm{ps}})^{-1} \Big)}{n_P\pi^P(B(x,h_n)) \vee n_Q\pi^Q(B(x,h_n))}\\ 
&\leq \frac{336\sigma^2 \Big(\overbar{\mathfrak{p}}\big\lVert\frac{\diff{\mu^P}}{\diff \pi^P}\big\rVert_{L^{\mathfrak{p}}(\pi^P)} (\gamma^P_{\mathrm{ps}})^{-1} \vee \overbar{\mathfrak{q}}\big\lVert\frac{\diff{\mu^Q}}{\diff \pi^Q}\big\rVert_{L^{\mathfrak{q}}(\pi^Q)} (\gamma^Q_{\mathrm{ps}})^{-1} \Big)}{n_P\pi^P(B(x,h_n)) \vee n_Q\pi^Q(B(x,h_n))},\\ 
\end{split}
\end{equation}
where for the first inequality we used the elementary inequality 
\[(a+b)^{-1}\one_{\{a+b > 0\}} \leq 2 ((1+a)^{-1} \wedge (1+b)^{-1}), \quad a,b \in \N_0,\] 
see the proof of Lemma 6 in \cite{pathak22}. Consequently, using the bounds \eqref{eq:i1} and \eqref{eq:i2} in \eqref{eq:biasvar}, we obtain 
\[\E\big[(\hat{f}_n(x) - f^\ast(x))^2\big] \leq L^2 h_n^{2\beta} + \mathfrak{C}(\gamma^{P}_{\mathrm{ps}},\gamma^Q_{\mathrm{ps}}, \mu^P,\mu^Q,\mathfrak{p},\mathfrak{q}) \frac{\sigma^2 + \lvert f^\ast(x)\rvert^2}{n} \frac{1}{\mu_n(B(x,h_n))}, \quad x \in \mathcal{X}, \]
from which the assertion follows by the Fubini--Tonelli theorem.
\end{proof}
\begin{remark} \label{rem:rate}
In the case without covariate shift, i.e., $P = Q$, $n_Q = 0$, following the lines of the proof we obtain the bound 
\[\E\Big[\big\lVert \hat{f}_n - f^\ast \big\rVert_{L^2(\pi^P)}^2 \Big] \leq  L^2h_n^{2\beta} + \mathfrak{C}^\prime  \frac{\pi^P(\lvert f^\ast \rvert^2) + \sigma^2}{n} \rho_{h_n}(\pi^P,\pi^P),\] 
for an explicit constant $\mathfrak{C}^\prime$ only depending on $\gamma_{\mathrm{ps}}^P, \mu^P,\mathfrak{p}$. If $\mathcal{X} \Subset \R^{\mathsf{d}}$ with diameter $D > 0$, then choosing $h_n = n^{-1/(2\beta+\mathsf{d})}$ we obtain 
\[\E\Big[\big\lVert \hat{f}_n - f^\ast \big\rVert_{L^2(\pi^P)}^2 \Big] \lesssim n^{-\frac{2\beta}{2\beta + \mathsf{d}}},\] 
since by \cite[Proposition 1]{pathak22} and Lemma \ref{lem:cov}, $\rho_{h_n}(\pi^P,\pi^P) \leq N_{\mathcal{X}}(h_n/2,d) \leq (1+ 4D/h_n)^{\mathsf{d}}$. Thus, without covariate shift, the usual minimax nonparametric regression rate for i.i.d.\ data is achieved for Markovian design with pseudo spectral gap.
\end{remark}

Theorem \ref{theo:upper} allows us to also study the prediction error 
\[\E\big[\big(\hat{f}_n(X^Q_{n_Q+m}) - f^\ast(X^Q_{n_Q + m})\big)^2\big], \quad m\in \N_0,\] 
i.e., the expected squared error at a location $X^Q_{n_Q+m}$ from the same Markov chain underlying the target data sample $\mathcal{S}_Q$, in terms of the mixing behavior of the chain $X^Q$. For the next statement, we suppose that for $R \in \{P,Q\}$, $\xi_i^R = \varsigma^R_i(X_i^R) \varepsilon^R_i$, where $\varsigma^R_i\colon \cX \to [-\sigma,\sigma]$ and the random vector $(\varepsilon^R_i, R \in \{P,Q\}, i \in \N_0)$ has independent components, is independent of $(X^P,X^Q)$ and $\E[\varepsilon_i^R] = 0, \E[(\varepsilon^R_i)^2] \leq 1$.

\begin{theorem} \label{theo:pred}
Assume the above conditions on the noise $(\xi^P,\xi^Q)$ and assume that $f^\ast \in \mathcal{H}(\beta,L)$  and $n_P \geq (\gamma^P_{\mathrm{ps}})^{-1}, n_Q \geq (\gamma^Q_{\mathrm{ps}})^{-1}$.
\begin{enumerate}[label = (\roman*), ref =(\roman*)]
\item \label{pred_psi} Suppose that $f^\ast \in L^2(\pi^Q)$ and that for some $m \in \N$ and all $x \in \cX$, $Q^m(x,\cdot) \ll \pi^Q$ with Radon--Nikodym derivative satisfying 
\begin{equation} \label{eq:rn}
\sup_{x \in \cX} \Big\lVert \frac{\diff{Q^m(x,\cdot)}}{\diff{\pi^Q}} \Big\rVert_{L^\infty(\pi^Q)} \leq M,
\end{equation}
for some $M > 0$. Then,  
\begin{align*}
\E\big[\big(\hat{f}_n(X^Q_{n_Q+m-1}) - f^\ast(X^Q_{n_Q + m-1})\big)^2\big] &\leq  M \E\Big[\big\lVert \hat{f}_n - f^\ast \big\rVert_{L^2(\pi^Q)}^2 \Big] \\
&\leq ML^2h_n^{2\beta} + M\mathfrak{C}  \frac{\pi^Q(\lvert f^\ast \rvert^2) + \sigma^2}{n} \rho_{h_n}(\mu_n,\pi^Q).
\end{align*}
\item \label{pred_uni} Suppose that $f^\ast$ is bounded and that for $R \in \{P,Q\}$, the random vectors $(\varepsilon^R_i)_{i \in \N_0}$ are uniformly sub-Gaussian, i.e., for some $\zeta_R > 0$ and $t \geq 0$, $\PP(\lvert \xi_i^R \rvert \geq t) \leq 2\exp(-t^2/\zeta_R^2)$. Also assume that $X^Q$ is geometrically ergodic, i.e., there exists constants $c_Q > 0, \kappa_Q \in (0,1)$ and a function $V\colon \cX \to \R_+$ such that 
\[\lVert Q^n(x,\cdot) - \pi^Q \rVert_{\mathrm{TV}} \leq c_Q V(x)\kappa_Q^n, \quad x \in \cX, n \in \N_0,\] 
and, moreover, $c^\prime_Q \coloneq \sup_{n \in \N} \E^{\mu^Q}[V(X^Q_n)] < \infty$.
Then, for some global constant $C > 0$ and $C_Q = c_Q c_Q^\prime$, it holds for any $m \in \N$,
\begin{align*}
&\E\big[\big(\hat{f}_n(X^Q_{n_Q+m-1}) - f^\ast(X^Q_{n_Q + m-1})\big)^2\big]\\
&\,\leq \E\Big[\big\lVert \hat{f}_n - f^\ast \big\rVert_{L^2(\pi^Q)}^2 \Big] +  4c_Q \kappa_Q^m \big(\lVert f^\ast \rVert_\infty^2 + \sigma^2 (\zeta_P \vee \zeta_Q)^2(1 + \log n_P \vee n_Q) \big) \\
&\,\leq L^2h_n^{2\beta} + \mathfrak{C}  \frac{\lVert f^\ast \rVert^2_\infty + \sigma^2}{n} \rho_{h_n}(\mu_n,\pi^Q) + 4C_Q \kappa_Q^m \big(\lVert f^\ast \rVert_\infty^2 + 2\sigma^2 C(\zeta_P \vee \zeta_Q)^2(1 + \log n_P \vee n_Q) \big).
\end{align*}
\end{enumerate}
\end{theorem}
\begin{proof} 
Let $\hat{f}_n(x) \eqcolon h(x,X_1,\ldots,X_n,\xi_1,\ldots,\xi_n)$, $\varepsilon = (\varepsilon_1,\ldots,\varepsilon_n) \coloneq (\varepsilon_0^P,\ldots,\varepsilon_{n_P-1}^P,\varepsilon_0^Q,\ldots,\varepsilon_{n_Q-1}^Q)$ and denote $(\varsigma_1,\ldots,\varsigma_n) \coloneq (\varsigma^P_0,\ldots,\varsigma^P_{n_P-1},\varsigma^Q_0,\ldots,\varsigma^Q_{n_Q-1})$. Then, by independence of $\varepsilon$ and $(X_1,\ldots,X_n)$, conditioning on $\varepsilon$ yields 
\begin{equation}\label{eq:pred1}
\begin{split}
&\E\big[\big(\hat{f}_n(X^Q_{n_Q+m}) - f^\ast(X^Q_{n_Q + m})\big)^2\big] \\ 
&\, = \E\Big[\E\big[\big(h(X^Q_{n_Q+m},X_1,\ldots,X_n,\varsigma(X_1)y_1,\ldots,\varsigma_n(X_n)y_n)- f^\ast(X^Q_{n_Q + m})\big)^2\big]\big\vert_{(\varepsilon_1,\ldots,\varepsilon_n) =(y_1,\ldots,y_n)} \Big].
\end{split}
\end{equation}
By definition of $(X_1,\ldots,X_n)$ we obtain for 
\[K(\diff{x_0} \times \cdots \times \diff{x_{n_Q+n_P-1}}) = \mu^P(\diff{x_0}) \prod_{i=1}^{n_P-1} P(x_{i-1},\diff{x_i}) \mu^Q(\diff{x_{n_P}}) \prod_{i=1}^{n_Q-1} Q(x_{i - 1 + n_P},\diff{x_{i + n_P}}),\] 
that
\begin{align*} 
&\E\big[\big(h(X^Q_{n_Q+m-1},X_1,\ldots,X_n,y_1,\ldots,y_n)- f^\ast(X^Q_{n_Q + m-1})\big)^2\big]\\
&\,= \int_{\cX^{n+1}} K(\diff{x_0} \times \cdots \times \diff{x_{n-1}}) Q^m(x_{n-1}, \diff{x})\,  \big(h(x,x_0,\ldots,x_{n-1},\varsigma_1(x_0)y_1,\ldots,\varsigma_n(x_{n-1})y_n) - f^\ast(x)\big)^2.
\end{align*} 
For case \ref{pred_psi}, \eqref{eq:rn} therefore yields
\begin{align*}
&\E\big[\big(h(X^Q_{n_Q+m-1},X_1,\ldots,X_n,y_1,\ldots,y_n)- f^\ast(X^Q_{n_Q + m-1})\big)^2\big]\\
&\,\leq M \int_{\cX^{n+1}} K(\diff{x_0} \times \cdots \times \diff{x_{n-1}}) \pi^Q(\diff{x})\,  \big(h(x,x_0,\ldots,x_{n-1},\varsigma_1(x_0)y_1,\ldots,\varsigma_n(x_{n-1})y_n) -f^\ast(x)\big)^2 \\ 
&\,= M \int_{\cX} \E\big[\big(h(x,X_1,\ldots,X_n,\varsigma_1(X_1)y_1,\ldots,\varsigma_n(X_n)y_n)- f^\ast(x)\big)^2\big] \, \pi^Q(\diff{x}).
\end{align*}
Thus, by \eqref{eq:pred1}, using once more independence of $\varepsilon$ and $(X_1,\ldots,X_n)$, 
\[\E\big[\big(\hat{f}_n(X^Q_{n_Q+m-1}) - f^\ast(X^Q_{n_Q + m-1})\big)^2\big] \leq M \E\Big[\big\lVert \hat{f}_n - f^\ast \big\rVert_{L^2(\pi^Q)}^2 \Big].\] 
Theorem \ref{theo:upper} now yields statement \ref{pred_psi}.

In case \ref{pred_uni}, first note that 
\begin{align*}
&\sup_{x \in \cX} \lvert h(x,,x_0,\ldots,x_{n-1}, \varsigma_1(x_0)y_1,\ldots,\varsigma_n(x_{n-1})y_n) - f^\ast(x)\rvert \\
&\,= \sup_{x \in \cX} \Big\lvert  \frac{\sum_{i=1}^n (f(x_{i-1}) + \varsigma_i(x_{i-1}) y_i) \one_{B(x,h_n)}(x_{i-1})}{\sum_{i=1}^n \one_{B(x,h_n)}(x_{i-1})} \one_{(0,\infty)}\Big(\sum_{i=1}^n \one_{B(x,h_n)}(x_{i-1}) \Big) - f^\ast(x) \Big\rvert^2\\
&\,\leq 4\lVert f^\ast \rVert_\infty^2 + 4\sigma^2 \max_{i=1,\ldots,n} y_i^2.
\end{align*}
Thus, 
\begin{align*} 
&g(y_1,\ldots,y_n) \coloneq \\
&\,\Big\lvert \int_{\cX^{n+1}} K(\diff{x_0} \times \cdots \times \diff{x_{n-1}}) Q^m(x_{n-1}, \diff{x})\,  \big(h(x,x_0,\ldots,x_{n-1},\varsigma_1(x_0)y_1,\ldots,\varsigma_n(x_{n-1})y_n) - f^\ast(x)\big)^2 \\ 
&\quad - \int_{\cX^{n+1}} K(\diff{x_0} \times \cdots \times \diff{x_{n-1}}) \pi^Q(\diff{x})\,  \big(h(x,x_0,\ldots,x_{n-1},\varsigma_1(x_0)y_1,\ldots,\varsigma_n(x_{n-1})y_n) - f^\ast(x)\big)^2 \Big\rvert \\
&\,\leq \int_{\cX^n} K(\diff{x_0} \times \cdots \times \diff{x_{n-1}}) \\
&\qquad  \Big\lvert \int_{\cX} (Q^m(x_{n-1}, \diff{x}) - \pi^Q(\diff{x}))\,  \big(h(x,x_0,\ldots,x_{n-1},\varsigma_1(x_0)y_1,\ldots,\varsigma_n(x_{n-1})y_n) - f^\ast(x)\big)^2 \Big\rvert \\
&\,\leq 4\big(\lVert f^\ast \rVert_\infty^2 + \sigma^2 \max_{i=1,\ldots,n} y_i^2\big) \int_{\cX^n} K(\diff{x_0} \times \cdots \times \diff{x_{n-1}}) \, \big\lVert Q^m(x_{n-1},\cdot) - \pi^Q \big\rVert_{\mathrm{TV}}\\
&\,\leq 4c_Q \kappa_Q^m \big(\lVert f^\ast \rVert_\infty^2 + \sigma^2 \max_{i=1,\ldots,n} y_i^2\big) \int_{\cX^{n}} K(\diff{x_0} \times \cdots \times \diff{x_{n-1}})\, V(x_{n-1})\\
&\, = 4c_Q \E^{\mu^Q}[V(X^Q_{n_Q})]\kappa_Q^m \big(\lVert f^\ast \rVert_\infty^2 + \sigma^2 \max_{i=1,\ldots,n} y_i^2\big)\\
&\leq 4C_Q \kappa_Q^m \big(\lVert f^\ast \rVert_\infty^2 + \sigma^2 \max_{i=1,\ldots,n} y_i^2\big),
\end{align*}
and we obtain by conditioning on $(\varepsilon_1,\ldots,\varepsilon_n)$ and triangle inequality
\begin{equation}\label{eq:pred_uni}
\begin{split} 
\Big\lvert \E\big[\big(\hat{f}_n(X^Q_{n_Q+m}) - f^\ast(X^Q_{n_Q + m})\big)^2\big] - \E\Big[\big\lVert \hat{f}_n - f^\ast \big\rVert_{L^2(\pi^Q)}^2 \Big] \Big\rvert &\leq \E[g(\varepsilon_1,\ldots,\varepsilon_n)]\\
&\leq 4C_Q \kappa_Q^m \big(\lVert f^\ast \rVert_\infty^2 + \sigma^2 \E\big[\max_{i=1,\ldots,n} \lvert \varepsilon_i\rvert ^2 \big]\big).
\end{split}
\end{equation}
Since the elements of $\varepsilon^P$ are uniformly sub-Gaussian, we have 
\[\PP(\lvert \varepsilon_i^P\rvert^2 \geq t) = \PP(\lvert \varepsilon_i^P\rvert \geq \sqrt{t}) \leq 2\exp(-t/\zeta^2_P), \quad t \geq 0,\] 
whence $(\lvert \varepsilon^P_i \rvert^2)_{i \in \N_0}$ is uniformly sub-exponential and \cite[Proposition 2.7.1]{vershy18} gives 
\[\E[\exp(\lambda \lvert \varepsilon_i^P \rvert^2)] \leq \exp(C\zeta^2_P\lambda), \quad \lambda \leq 1/(C\zeta^2_P),\] 
for some global constant $C > 0$. Therefore, a union bound gives 
\[\E\Big[\exp\big((C\zeta_P^2)^{-1} \max_{i=1,\ldots,n_P}\lvert \varepsilon_i^P \rvert^2\big) \Big] \leq n_P\E\big[\exp((C\zeta_P^2)^{-1}\lvert \varepsilon^P_1 \rvert^2)\big] \leq \mathrm{e} n_P,\] 
such that Jensen's inequality implies 
\[\E\big[\max_{i=0,\ldots,n_P-1} \lvert\varepsilon^P_i\rvert^2\big] \leq C\zeta_P^2(1 + \log n_P).\]
Similarly, we obtain 
\[\E\big[\max_{i=0,\ldots,n_Q-1} \lvert\varepsilon^Q_i\rvert^2\big] \leq C\zeta_Q^2(1 + \log n_Q),\]
whence \eqref{eq:pred_uni} yields 
\[\Big\lvert \E\big[\big(\hat{f}_n(X^Q_{n_Q+m}) - f^\ast(X^Q_{n_Q + m})\big)^2\big] - \E\Big[\big\lVert \hat{f}_n - f^\ast \big\rVert_{L^2(\pi^Q)}^2 \Big] \Big\rvert \leq 4C_Q \kappa_Q^m \big(\lVert f^\ast \rVert_\infty^2 + 2\sigma^2 C(\zeta_P \vee \zeta_Q)^2(1 + \log n_P \vee n_Q) \big).\]
Statement \ref{pred_uni} now follows by triangle inequality and Theorem \ref{theo:upper}.
\end{proof}
\begin{remark} 
\begin{enumerate}[label = (\alph*), ref =(\alph*)]
\item Condition \eqref{eq:rn} represents an especially pronounced  mixing condition on the chain $X^Q$. 
A sufficient condition for \eqref{eq:rn} to hold is given by 
\begin{equation}\label{eq:rn2}
\forall x,x^\prime \in \cX, A \in \mathcal{B}(\cX): Q^m(x,A) \leq MQ^m(x^\prime,A),
\end{equation} 
since this implies $M\pi^Q(A) = M \int_{\cX} Q^m(x^\prime,A)\, \pi^Q(\diff{x^\prime}) \geq Q^m(x,A)$ for any $x \in \cX$.
Condition \eqref{eq:rn2} appears, e.g., in \cite{ellis88} and implies $\psi$-mixing of the chain $X^Q$, cf.\ \cite[Remark 1.4]{bradley97}.
\item In case \ref{pred_uni}, condition $\sup_{n \in \N} \E^{\mu^Q}[V(X_n^Q)] < \infty$ is trivially satisfied for any initial distribution $\mu^Q$ when $X^Q$ is uniformly ergodic, i.e., $\lVert V \rVert_\infty < \infty$. Moreover, it also holds for $V$-uniformly ergodic chains, i.e., chains satisfying
\[\sup_{\lvert f \rvert \leq V} \lvert \E^x[f(X^Q_n)] - \mu^Q(f) \rvert \lesssim V(x) \kappa_Q^n, \quad x \in \cX, n \in \N_0,\] 
and initial distributions $\mu^Q \ll \pi^Q$ with $\diff{\mu^Q}/\diff{\pi^Q} \in L^\infty(\pi^Q)$, cf.\ \cite[Chapter 16]{mt09}.
\end{enumerate}
\end{remark}

\section{Explicit rates}\label{sec:explicit}
We now discuss specific models for which the dissimilarity $\rho_{h_n}(\mu_n,\pi^Q)$ and hence the convergence rate of the NW estimator can be explicitly controlled based on the topology of the state space.

\subsection{Finite Markov chains}
Let us first consider the simplest case when $X^P,X^Q$ are finite Markov chains, i.e., $\lvert \mathcal{X} \rvert \eqcolon K \in \N$ with $\mathcal{X} = \{x_1,\ldots,x_K\}$. Assume that the transition matrices $P = (p_{i,j})_{i,j \in [K]},Q = (q_{i,j})_{i,j \in [K]}$ are aperiodic and irreducible and let $\bm{\pi}^P =(\pi^P_i)_{i \in [K]},\bm{\pi}^Q =(\pi^Q_i)_{i \in [K]}$ be the invariant distribution vectors for $P$ and $Q$ respectively. For simplicity, let us consider the case when $P,Q$ are reversible, i.e., $\pi^P_i p_{i,j} = \pi^P_i p_{j,i}$ and $\pi^Q_i q_{i,j} = \pi^Q_i q_{j,i}$ for all $i,j \in [K]$. The following discussion can be easily extended to the non-reversible case by considering the multiplicative symmetrizations $PP^\ast, QQ^\ast$ instead, provided that these matrices are irreducible. Given reversibility, all eigenvalues of $P,Q$ are real and since both matrices are irreducible and aperiodic, the Perron--Frobenius theorem implies that the second largest absolute eigenvalues $\lambda(P) = \vertiii{ P - \Pi^P}_{\pi^P}, \lambda(Q) = \vertiii{Q - \Pi^Q}_{\pi^Q}$ are strictly dominated by the largest eigenvalue $\lambda = 1$. Thus, $P$ and $Q$ have absolute spectral gaps $1- \lambda(P) > 0$ and $1 - \lambda(Q) > 0$ and therefore  pseudo spectral gaps $\gamma^P_{\mathrm{ps}} \geq 1 - \lambda(P)^2, \gamma^Q_{\mathrm{ps}} \geq 1 - \lambda(Q)^2$, respectively (see \eqref{eq:abs_pseudo}).

Since $\mathcal{X}$ is finite, any  function $g\colon \cX \to \R$ is $\overbar{L}$-Lipschitz continuous with  Lipschitz constant $\overbar{L} \coloneq \max_{i,j \in [K]} \lvert g(x_i) - g(x_j) \rvert/(\min_{i,j \in [K]^2, i \neq j} d(x_i,x_j))$ and the nonparametric regression problem effectively becomes a parametric one since $f^\ast = \sum_{i=1}^K a_i \one_{\{x_i\}}$ with $a_i = f^\ast(x_i)$. Correspondingly, for any $h_n < \min_{i,j\in [K]^2, i \neq j} d(x_i,x_j) \eqcolon \delta$, the NW estimator $\hat{f}_n$ is a simple weighted empirical  estimator given by 
\[\hat{f}_n(x) \coloneq \frac{\sum_{i=1}^n Y_i \one_{\{X_i = x\}}}{\#\{i \in [n]: X_i = x\}} \one_{\{\#\{i \in [n]: X_i = x\} > 0\}}, \quad x \in \mathcal{X}.\]
In this simple setting, our previous results imply that even under covariate shift we always recover a parametric convergence rate $\mathfrak{c}/\sqrt{n}$ for the generalization risk, where the deviation of $\pi^P$ from the target distribution $\pi^Q$ is only represented in the multiplicative constant $\mathfrak{c}$ in terms of the maximum of the Radon--Nikodym derivative $\diff \pi^Q/\diff{\pi^P}$. Let 
\[n_{\mathrm{eff}} \coloneq \frac{n_P}{\max_{i \in [K]} \pi^Q_i/\pi^P_i} + n_Q\] 
be the \textit{effective sample size} of the covariate shift model.

\begin{proposition} 
Let $X^P,X^Q$ be the independent finite Markov chains from above with initial distributions $\mu^P$ and $\mu^Q$, resp., and denote $\kappa(P) \coloneq \max_{i \in [K]} \tfrac{\mu^P_i}{\pi^P_i}$, $\kappa(Q) \coloneq \max_{i \in [K]} \tfrac{\mu^Q_i}{\pi^Q_i}$. Assume $n_P \geq (1 - \lambda(P)^2)^{-1}, n_Q \geq (1- \lambda(Q)^2)^{-1}$. Then, for 
\begin{align*} 
\mathfrak{C}^\prime(P,Q,\mu^P,\mu^Q) \coloneq  12 \max\big\{&3 \kappa(P)\kappa(Q) (1 - (\lambda(P) \vee \lambda(Q))^2)^{-1}, \\
&28(\kappa(P)(1-\lambda(P)^2)^{-1} \vee \kappa(Q)(1-\lambda(Q)^2)^{-1})\big\},
\end{align*}
and the choice $h_n = c(n_{\mathrm{eff}}^{-1/2}\wedge \delta)$ for some $c \in (0,1)$, 
we have the following bound for the $(P,Q)$-covariate shift model:
\[\E\Big[\big\lVert \hat{f}_n - f^\ast \big\rVert_{L^2(\pi^Q)}^2 \Big] \leq \frac{c^2\overbar{L}^2 + K \max_{i \in [K]} \frac{\pi^Q_i}{\pi^P_i} \mathfrak{C}^\prime(P,Q,\mu^P,\mu^Q)(\lVert f^\ast \rVert_\infty^2 + \sigma^2)}{n_{\mathrm{eff}}}.\] 
\end{proposition}
\begin{proof} 
First note that  for $K \in \{P,Q\}$ it holds  $\mu^K \in L^\infty(\pi^K)$ with $\lVert \diff{\mu^K \slash \diff \pi^K } \rVert_{L^\infty(\pi^K)} = \kappa(K)$ and hence 
\[\mathfrak{C}(\gamma^P_{\mathrm{ps}},\gamma^Q_{\mathrm{ps}},\mu^P,\mu^Q,\infty,\infty) \leq \mathfrak{C}(1 - \lambda(P)^2, 1 - \lambda(Q)^2,\mu^P,\mu^Q,\infty,\infty) = \mathfrak{C}^\prime(P,Q,\mu^P,\mu^Q).\]
Using $f \in \mathcal{H}(1,\overbar{L})$, Theorem \ref{theo:upper} therefore gives the bound 
\begin{equation}\label{eq:finite}
\E\Big[\lVert \hat{f}_n - f^\ast \rVert_{L^2(\pi^Q)}^2 \Big] \leq \frac{c^2\overbar{L}^2}{n_{\mathrm{eff}}} + \mathfrak{C}^\prime(P,Q,\mu^P,\mu^Q) \frac{\lVert f^\ast \rVert_\infty^2 + \sigma^2}{n} \rho_{h_n}(\mu_n,\pi^Q).
\end{equation}
Since $h_n < \delta$ we have for any $x \in \cX$,
\begin{align*}
\mu_n(B(x,h_n)) = \frac{n_P}{n}\pi^P(\{x\}) + \frac{n_Q}{n} \pi^Q(\{x\}) &\geq \frac{n_P \min_{i \in [K]} \frac{\pi^P_i}{\pi^Q_i} + n_Q}{n}\pi^Q(\{x\})\\
&= \frac{n_P \min_{i \in [K]} \frac{\pi^P_i}{\pi^Q_i} + n_Q}{n}\pi^Q(B(x,h_n)).
\end{align*}
Hence, it follows from \cite[Proposition 1]{pathak22} and $h_n/2 < \delta$
\begin{align*}
\rho_{h_n}(\mu_n,\pi^Q) \leq \frac{n}{n_P \min_{i \in [K]} \frac{\pi^P_i}{\pi^Q_i} + n_Q} N(h_n/2,d) &= \frac{n K}{n_{\mathrm{eff}}},
\end{align*}
which combined with \eqref{eq:finite} yields the assertion.
\end{proof}

\subsection{Markov chains with invariant distributions belonging to an $\alpha$-family}
Let us first extend the definition of $\alpha$-families of source-target pairs from \cite{pathak22} to higher dimensions, which for bounded Euclidean subspaces can be understood as a generalization of source-target pairs of probability distributions with transfer exponent introduced in \cite{kpotufe21}. Denote by $\mathcal{P}(\cX,\cB(\cX))$ the family of probability distributions on the measurable space $(\cX,\cB(\cX))$.

\begin{definition} 
Suppose that $\cX \subset \R^{\mathsf{d}}$ is bounded with diameter $D > 0$. For $\alpha \geq \mathsf{d}$ and $C \geq 1$, let 
\[\mathcal{D}(\alpha,C) \coloneq \Big\{(\PP,\QQ) \in \mathcal{P}(\cX,\cB(\cX))^2 : \sup_{0 < h \leq D} \big(\tfrac{h}{D}\big)^{\alpha} \rho_h(\PP,\QQ) \leq C \Big\}\] 
and for $\alpha \in (0,\mathsf{d})$, $\alpha^\prime \in [0,\alpha]$ and $C \geq 1$ define 
\[\mathcal{D}^\prime(\alpha,\alpha^\prime,C) \coloneq \Big\{(\PP,\QQ) \in \mathcal{P}(\cX,\cB(\cX))^2 : \sup_{0 < h \leq D} \big(\tfrac{h}{D}\big)^{\alpha} \rho_h(\PP,\QQ) \leq C, \, \sup_{0 < h \leq D} \big(\tfrac{h}{D}\big)^{\alpha^\prime} \rho_h(\QQ,\QQ) \leq C \Big\}.\]
\end{definition}

Typical examples of source-target distributions $(\PP,\QQ)$ that belong to $\mathcal{D}^\prime(\alpha,\alpha^\prime,C)$ for  $\alpha < \mathsf{d}$ are such that the support of $\QQ$ is contained in some subset of $\mathcal{X}$ that is isometric to a lower-dimensional subspace $\mathcal{X}_Q$ of $\R^{\mathsf{d}}$, see in particular Example \ref{ex:alpha} in Appendix \ref{app:exalpha}. This explains the second condition on the growth of $\rho_h(\mathbb{Q}, \QQ)$ as $h \downarrow 0$, which in such cases can be controlled by the covering number growth of a lower-dimensional subspace. More precisely, we have the following for the typical situation when the $\varepsilon$-covering number scales as $\varepsilon^{-\mathsf{d}_Q}$.
\begin{lemma} 
Suppose that $\QQ$ is supported on $\mathcal{X}_Q \subset \R^{\mathsf{d}}$ with $N_{\mathcal{X}_Q}(\varepsilon,d\vert_{\mathcal{X}_Q}) \asymp \varepsilon^{-\mathsf{d}_Q}$, $\mathsf{d}_Q \leq \mathsf{d}, \varepsilon \in (0,1]$. If $\QQ\vert_{\mathcal{X}_Q}$ is a doubling measure on $\mathcal{X}_Q$, i.e., for some $C > 0$ and any $x \in \cX_Q,r > 0$, 
\[\QQ(B_d(x,2r) \cap \mathcal{X}_Q) \leq C\QQ(B_d(x,r) \cap \mathcal{X}_Q),\]
we have $\rho_h(\mathbb{Q},\mathbb{Q}) \asymp h^{-\mathsf{d}_Q}$ for $h \in (0,1]$.
\end{lemma}
\begin{proof} 
Since the packing number $N^{\mathrm{pack}}_{\mathcal{X}_Q}(\varepsilon,d\vert_{\mathcal{X}_Q})$ satisfies $N^{\mathrm{pack}}_{\mathcal{X}_Q}(\varepsilon,d\vert_{\mathcal{X}_Q}) \geq N_{\mathcal{X}_Q}(2\varepsilon,d\vert_{\mathcal{X}_Q})$, we can find $H \gtrsim h^{-\mathsf{d}_Q}$ disjoint balls $B_{d\vert\mathcal{X}_Q}(x_i,\varepsilon) \subset \mathcal{X}_Q$ with centres $x_i \in \cX_Q$, s.t.\ $\bigcup_{i=1}^H B_{d\vert_{\mathcal{X}_Q}}(x_i,\varepsilon) \subset \mathcal{X}_Q$. Since $\QQ$ is supported on $\mathcal{X}_Q$ and is a doubling measure thereon, it follows 
\begin{align*} 
\rho_h(\QQ,\QQ) = \int_{\mathcal{X}_Q} \frac{1}{\QQ(B_d(x,h) \cap \mathcal{X}_Q)} \, \QQ(\diff{x}) &\geq \sum_{i=1}^H \int_{B_d(x_i,h) \cap \mathcal{X}_Q}\frac{1}{\QQ(B_d(x,h) \cap \mathcal{X}_Q)} \, \QQ(\diff{x}) \\ 
&\geq \sum_{i=1}^H \frac{\QQ(B_d(x_i,h) \cap \mathcal{X}_Q)}{\QQ(B_d(x_i,2h) \cap \mathcal{X}_Q)}\\
&\geq C^{-1}H \gtrsim h^{-\mathsf{d}_Q},
\end{align*}
where we used, that by triangle inequality, for any $x \in B_d(x_i,h)$ we have $B_d(x,h) \subset B_d(x_i,2h)$. Complementing this lower bound with the upper bound $\rho_h(\QQ,\QQ) \lesssim h^{-\mathsf{d}_Q}$, implied by \cite[Proposition 1]{pathak22}, finishes the proof.
\end{proof}

Source-target distribution combinations belonging to $\mathcal{D}^\prime(\alpha.\alpha^\prime,C)$ can be particularly desirable for transfer from $\pi^P$ to $\pi^Q$. 
In fact, we  have the following result that recovers \cite[Corollary 1]{pathak22} as a special case for i.i.d.\ samples and $\cX = [0,1]$. 

\begin{proposition} \label{prop:alpha}
Let $\cX \Subset \R^{\mathsf{d}}$ with diameter $0< D < \infty$ and a $(P,Q)$-Markov covariate shift regression model with $f^\ast \in \mathcal{H}(\beta,L)$ s.t.\ $\pi^Q(\lvert f^\ast\rvert^2) \leq M$ and $n_P \geq (\gamma^P_{\mathrm{ps}})^{-1}, n_Q \geq (\gamma^Q_{\mathrm{ps}})^{-1}$ be given. Define 
\begin{align*}
\mathfrak{L} &= \mathfrak{L}(\gamma^P_{\mathrm{ps}},\gamma^Q_{\mathrm{ps}},\mu^P,\mu^Q,\mathfrak{p},\mathfrak{q},\sigma^2,\alpha,\mathsf{d},C,D,M)\\
&\coloneq L^2 + 2((1+ 2^{\mathsf{d}})\vee C)(1 \vee D^{\alpha \vee \mathsf{d}})\mathfrak{C}(\lambda^P_{\mathrm{ps}}, \lambda^Q_{\mathrm{ps}}, \mu^P,\mu^Q,\mathfrak{p},\mathfrak{q})(M^2+ \sigma^2),
\end{align*}
and the bandwidth $h_n \coloneq \big(n_Q + n_P^{(2\beta+ \mathsf{d})/(2\beta + \alpha)}\big)^{-1/(2\beta + \mathsf{d})},$ where we assume $n$ large enough s.t.\ $h_n \leq  D$.
\begin{enumerate}[label = (\roman*), ref = (\roman*)]
\item \label{prop:alpha1} If $(\pi^P,\pi^Q) \in \mathcal{D}(\alpha,C)$ for some $\alpha \geq \mathsf{d}, C \geq 1$, then
\[\E\Big[\big\lVert \hat{f}_n -f^\ast \big\rVert^2_{L^2(\pi^Q)}\Big] \leq \mathfrak{L} \Big({n_P^{\frac{2\beta + \mathsf{d}}{2\beta + \alpha}}} + n_Q\Big)^{-\frac{2\beta}{2\beta + \mathsf{d}}}.\]
\item \label{prop:alpha2} If $(\pi^P,\pi^Q) \in \mathcal{D}(\alpha,\alpha^\prime,C)$ for some $\alpha < \mathsf{d}, \alpha^\prime \leq \alpha, C \geq 1$, then  
\[\E\Big[\big\lVert \hat{f}_n -f^\ast \big\rVert^2_{L^2(\pi^Q)}\Big] \leq \mathfrak{L} \Big(n_P^{\frac{2\beta + \alpha^\prime}{2\beta + \alpha}} + n_Q\Big)^{-\frac{2\beta}{2\beta + \alpha^\prime}}.\]
\end{enumerate}
\end{proposition}
\begin{proof}
Let $\zeta(\alpha,\mathsf{d}) \coloneq \mathsf{d} \one_{[d,\infty)}(\alpha) + \alpha^\prime \one_{(0,\mathsf{d})}(\alpha)$ and $\mathfrak{C} = \mathfrak{C}(\lambda^P_{\mathrm{ps}}, \lambda^Q_{\mathrm{ps}}, \mu^P,\mu^Q,\mathfrak{p},\mathfrak{q})$. By definition of $\mathcal{D}(\alpha,C)$, $\mathcal{D}(\alpha,\alpha^\prime,C)$ and, in case $\alpha \geq \mathsf{d}$, using $\rho_{h_n}(\pi^Q,\pi^Q) \leq N(h_n/2,d) \leq (1 + 2D/h_n)^{\mathsf{d}}$ we have
\begin{align*} 
\rho_h(\mu_n,\pi^Q) &= \int_{\cX} \frac{n}{n_P \pi^P(B(x,h_n)) + n_Q\pi^Q(B(x,h_n))} \, \pi^Q(\diff{x}) \\
&\leq n\min\Big\{\frac{1}{n_P} \rho_{h_n}(\pi^P,\pi^Q), \frac{1}{n_Q} \rho_{h_n}(\pi^Q, \pi^Q)  \Big\}\\
&\leq n\min\Big\{\frac{CD^\alpha}{n_Ph_n^\alpha}, \frac{(1+2D/h_n)^{\mathsf{d}} \one_{[\mathsf{d},\infty)}(\alpha) + C\one_{(0,\mathsf{d})}(\alpha)}{n_Q} \Big\}\\
&\leq ((1+ 2^{\mathsf{d}})\vee C)(1 \vee D^{\alpha \vee \mathsf{d}}) n\min\Big\{\frac{1}{n_Ph^\alpha_n}, \frac{1}{n_Q h^{\zeta(\alpha,\mathsf{d})}_n} \Big\}.
\end{align*}
Hence, Theorem \ref{theo:upper} implies 
\begin{equation}\label{eq:bound_alpha}
\E\Big[\lVert \hat{f}_n - f^\ast \rVert_{L^2(\pi^Q)}^2 \Big] \leq  L^2h_n^{2\beta} + ((1+ 2^{\mathsf{d}})\vee C)(1 \vee D^{\alpha \vee \mathsf{d}})\mathfrak{C} \frac{\pi^Q(\lvert f^\ast \rvert^2) + \sigma^2}{n_P h_n^{\alpha} \vee n_Q h_n^{\zeta(\alpha,\mathsf{d})}} .
\end{equation}
Observe now that 
\begin{align*} 
\max\Big\{n_P\Big(n_Q^{-\frac{\alpha}{2\beta + \zeta(\alpha,\mathsf{d})}} \wedge n_P^{-\frac{\alpha}{2\beta + \alpha}} \Big), n_Q\Big(n_Q^{-\frac{\zeta(\alpha,\mathsf{d})}{2\beta + \zeta(\alpha,\mathsf{d})}} \wedge n_P^{-\frac{\zeta(\alpha,\mathsf{d})}{2\beta + \alpha}} \Big)  \Big\} &\geq \begin{cases}  n_Q^{\frac{2\beta}{2\beta+\zeta(\alpha,\mathsf{d})}}, & n_Q > n_P^{\frac{2\beta + \zeta(\alpha,\mathsf{d})}{2\beta + \alpha}},\\ n_P^{\frac{2\beta}{2\beta+\alpha}}, & n_Q \leq n_P^{\frac{2\beta + \zeta(\alpha,\mathsf{d})}{2\beta + \alpha}}\end{cases}\\ 
&= n_P^{\frac{2\beta}{2\beta+\alpha}} \vee n_Q^{\frac{2\beta}{2\beta+\zeta(\alpha,\mathsf{d})}}.
\end{align*}
Hence, for 
\[h_n \coloneq \big(n_Q + n_P^{(2\beta+ \zeta(\alpha,\mathsf{d}))/(2\beta + \alpha)}\big)^{-\frac{1}{2\beta + \zeta(\alpha,\mathsf{d})}} \geq \frac{1}{2}\min\Big\{n_Q^{{-\frac{1}{2\beta + \zeta(\alpha,\mathsf{d})}}}, n_P^{-\frac{1}{2\beta + \alpha}} \Big\},\]
it follows 
\[\frac{1}{{n_P h_n^{\alpha} \vee n_Q h_n^{\zeta(\alpha,\mathsf{d})}}} \leq 2 n_P^{-\frac{2\beta}{2\beta+\alpha}} \wedge 2n_Q^{-\frac{2\beta}{2\beta+\zeta(\alpha,\mathsf{d})}} \leq 4 \Big(n_P^{\frac{2\beta + \zeta(\alpha,\mathsf{d})}{2\beta + \alpha}} + n_Q\Big)^{-\frac{2\beta}{2\beta + \zeta(\alpha,\mathsf{d})}}. \]
Plugging $h_n$ into \eqref{eq:bound_alpha} now yields the claim.
\end{proof}
\begin{remark} 
Let $\mathfrak{K}(\Gamma,C,\alpha)$ be a collection of Markov kernel pairs $(P,Q)$ on $\cX \Subset \R^{\mathsf{d}}$ with diameter $D$ endowed with the restriction of the maximum metric $d_\infty$ such that for some $\Gamma \in (0,1]$ we have pseudo spectral gaps  $\gamma^P_{\mathrm{ps}} \vee \gamma^Q_{\mathrm{ps}} \geq \Gamma$ and assume that the independent Markov chains $X^P, X^Q$ are stationary. Let $n_P \wedge n_Q \geq \Gamma^{-1}$ and 
define 
\[\mathcal{F}(\beta,L) \coloneq \mathcal{H}(\beta,L) \cap \big\{f\colon \cX \to \R: \lVert f \rVert_\infty \leq L \big\}.\] 
Then, Proposition \ref{prop:alpha} yields the uniform upper bound 
\[\inf_{\hat{f}} \sup_{(P,Q) \in \mathfrak{K}(\Gamma,C,\alpha)} \sup_{f^\ast \in \mathcal{F}(\beta,L)} \E_{f^\ast}\Big[\big\lVert \hat{f} -f^\ast \big\rVert_{L^2(\pi^Q)}\Big] \leq C(\alpha,\mathsf{d},L,\Gamma,\sigma^2,D) \Big(n_P^{\frac{2\beta +\mathsf{d}}{2\beta +\alpha}} + n_Q \Big)^{-\frac{2\beta}{2\beta + \mathsf{d}}}.\]
For the special case $\cX = [0,1]$ and some specified constant $n_{\ell}(\sigma,L,C,\alpha,\beta)$, for $n_P \vee n_Q \geq n_{\ell}$ we get by \cite[Theorem 2]{pathak22}  the accompanying lower bound 
\[\inf_{\hat{f}} \sup_{(P,Q) \in \mathfrak{K}(\Gamma,C,\alpha)} \sup_{f^\ast \in \mathcal{F}(\beta,L)} \E_{f^\ast}\Big[\big\lVert \hat{f} -f^\ast \big\rVert_{L^2(\pi^Q)}\Big] \geq c(\alpha,L,\sigma^2) \Big(n_P^{\frac{2\beta +\mathsf{d}}{2\beta +\alpha}} + n_Q \Big)^{-\frac{2\beta}{2\beta + \mathsf{d}}},\]
by choosing the independence kernels $P(x,\cdot) \coloneq \pi^P, Q(x,\cdot) \coloneq \pi^Q$ with pseudo spectral gap $1 \geq \Gamma$ for the hard distributions $(\pi^P,\pi^Q)$ specified in \cite[Section 4.2]{pathak22}. Hence, for the scalar case and $\alpha \geq 1$, part \ref{prop:alpha1} yields a minimax-optimal convergence rate. As remarked in \cite{pathak22}, with some effort, their lower bound proof for the scalar i.i.d.\ covariate shift regression  may be extended to higher dimensions, i.e., $\cX = [0,1]^{\mathsf{d}}$, by first constructing a hard distribution pair $(\pi^P,\pi^Q)$ as in the lower bound for transfer learning from \cite[Section 4]{kpotufe21} and then follow a similar program as in the case $\mathsf{d} = 1$. This then also implies the minimax optimality of the upper bound given in part \ref{prop:alpha1} in any dimension $\mathsf{d} \geq 1$. For the scalar case, the same discussion applies to the minimax optimality of the bound provided in part \ref{prop:alpha2} for $\alpha <1, \alpha^\prime = 0$. In higher dimensions optimality for $\alpha < \mathsf{d}$ is still open in the i.i.d.\ case.
\end{remark}

Proposition \ref{prop:alpha} shows that when $\alpha > \mathsf{d}$ and $n_Q/n \to 0$, the generalization risk is significantly larger than in the situation without covariate shift, see Remark \ref{rem:rate}, and prediction becomes polynomially harder with increasing order of the index $\alpha$.  On the other hand when $\alpha$ is  smaller than $\mathsf{d}$, estimating $f^\ast$ based on Markov data $X^P$ yields faster estimation rates than the worst-case minimax rate $n^{-2\beta/(2\beta + \mathsf{d})}$ without covariate shift. 
Note also that the convergence rate $n^{-2\beta/(2\beta + \alpha)}$ improves to $n^{-2\beta/(2\beta + \alpha^\prime)}$ if we have access to a pure training sample coming from $X^Q$, i.e., when no transfer is necessary w.r.t.\ $\pi^Q$.
The $\alpha$-index is therefore a suitable measure for the similarity between source-target pairs and the associated complexity of the estimation task for regression based on Markov data exhibiting a covariate shift.  \smallskip

For specific ergodic Markov models the transition kernel is often specified but the invariant distribution is unknown. To bring the previous result to life we therefore seek an appropriate class of Markov kernel pairs s.t.\ the corresponding invariant distributions belong to some $\alpha$-family. To this end we focus on  uniformly ergodic Markov chains as a subclass of spectral gap Markov chains. Suppose therefore that $X^P$ is a uniformly ergodic Markov chain, i.e., for some constants $c_P > 0$ and $\kappa_P \in (0,1)$ it holds 
\[\lVert P^n(x,\cdot) - \pi^P \rVert_{\mathrm{TV}} \leq c_P \kappa_P^n.\]
As discussed in Section \ref{sec:specgap}, this is equivalent to the Doeblin recurrence condition 
\[P^{m_P}(x,\cdot) \geq \varepsilon_P \nu^P(\cdot), \quad x \in \cX,\] 
for some $m_{P} \in \N, \varepsilon_{P} \in (0,1)$ explicitly related to $c_{P}, \kappa_{P}$ by \eqref{eq:doeblin_rate} and a probability measure $\nu^{P}$ on $(\cX,\cB(\cX))$. We call $\nu^P$ a \textit{minorizing measure} of $P$. Recall that $P$ has pseudo spectral gap bounded by the mixing time via
\[\gamma^P_{\mathrm{ps}} \geq \frac{1}{2 t^P_{\mathrm{mix}}} > 0. \] 
We now introduce the new concept of a \textit{transfer exponent} of the transition kernels $P,Q$, which will play an analogous role to the transfer exponent of a pair of source and target distributions given in \cite{kpotufe21} in the context of an i.i.d.\ classification problem. 

\begin{definition} 
We say that the kernels $(P,Q)$ have transfer exponent $\gamma \geq 0$ with constant $C > 0$ and radius $\overbar{h}$ if there exists a minorizing measure $\nu^P$ of $P$ and $m_Q \in \N$ s.t.\ for some closed set $\cX_Q \subset \cX$ with the property $\supp Q^{m_Q}(x,\cdot) \subset \cX_Q$ for all $x \in \cX$, we have 
\[\forall x,y \in \cX_Q, h \in (0,\overbar{h}]\colon \quad \nu^P(B(x,h)) \geq C\Big(\frac{h}{\overbar{h}}\Big)^\gamma Q^{m_Q}(y,B(x,h)).\] 
The class of such kernel pairs is denoted by $\mathcal{T}(\gamma,C,\overbar{h})$.
\end{definition}
If $m_Q = 1$, the set $\cX_Q$ is absorbing for the Markov chain $X^Q$ and its state space can be effectively confined to $\cX_Q$. We do not necessarily require the Markov chain $X^Q$ to be uniformly ergodic as well, although this assumption is quite plausible for concrete models as demonstrated in the illustrative examples for Markov kernels with transfer exponent considered next. We start with a case when the transfer exponent is $0$.

\begin{example}[Bounded kernel density ratio]
Suppose that for any $x \in \cX$ we have $Q(x,\cdot) \ll \nu^P$ and $\sup_{x \in \cX} \tfrac{\diff{Q(x,\cdot)}}{\diff\nu^P} \leq C < \infty$. Then $Q(y,B(x,h)) \leq C\nu^P(B(x,h))$ for any $x,y \in \cX$, i.e., $(P,Q) \in \mathcal{T}(0,C^{-1},\overbar{h})$ for any $\overbar{h} > 0$. E.g., let $\cX \subset \R^{\mathsf{d}}$ be bounded with $\lebesgue(\cX) > 0$, $P(x,\cdot) \sim \lebesgue\vert_{\cX}$ and $Q(x,\cdot) \ll P(x,\cdot)$ with 
\[0 < c \leq \inf_{x \in \cX} \tfrac{\diff{P(x,\cdot)}}{\diff\lebesgue\vert_{\cX}} \leq \sup_{x \in \cX} \tfrac{\diff{P(x,\cdot)}}{\diff\lebesgue\vert_{\cX}} \leq C< \infty \] 
and $\sup_{x \in \cX} \tfrac{\diff{Q(x,\cdot)}}{\diff{P(x,\cdot)}} \leq C^\prime < \infty$. Then, $\nu^P \coloneq \lebesgue\vert_{\cX} / \lebesgue(\cX)$ is a minorizing measure for $P$ since $P(x,\cdot) \geq c\lebesgue(\cX) \tfrac{\lebesgue\vert_{\cX}}{\lebesgue(\cX)}$ and, moreover, 
\[\sup_{x \in \cX} \frac{\diff{Q(x,\cdot)}}{\diff{\nu^P}} = \sup_{x \in \cX} \frac{\diff{Q(x,\cdot)}}{\diff P(x,\cdot)}\frac{\diff{P(x,\cdot)}}{\diff{\nu^P}} \leq \lebesgue(\cX)CC^\prime < \infty. \]
Hence, from  above, $(P,Q) \in \mathcal{T}(0,(\lebesgue(\cX)CC^\prime)^{-1},\overbar{h})$ for any $\overbar{h} > 0$.
\end{example}

Next, we construct explicit models with strictly positive transfer exponent.
\begin{example} \label{ex:beta}
Let $\mathrm{B}(\alpha,\beta)$ be the beta distribution, for $\gamma \geq 0$ let $\psi_\gamma$ be the density of $\mathrm{B}(1+\gamma,1)$, i.e., $\psi_\gamma(x) = (1+\gamma) x^\gamma$, $x \in [0,1]$, and let $F_\gamma$ be the corresponding CDF with left inverse $F^{-1}_\gamma$. For independent vectors $(U_i)_{i \in \N} \overset{\mathrm{iid}}{\sim} \mathrm{U}([0,1])$, $(\tilde{U}_i)_{i \in \N} \overset{\mathrm{iid}}{\sim} \mathrm{U}([0,1])$ consider the Markov chains $(X^P_n)_{n \in \N_0}$, $(X^Q_n)_{n \in \N_0}$ given by $X^P_{n+1} = F^{-1}_{\gamma_P + X_n^P}(U_n)$ and $X^Q_{n+1} = F^{-1}_{\gamma_Q + X_n^Q}(\tilde{U}_n)$ with $0\leq \gamma_Q < \gamma_P < \infty$ and  independent initial values $X_0^P \in [0,1]$ and $X^Q_0 \in [0,1]$. Then, $(X^P_{n+1} \mid X_n^P = x) \sim \mathrm{B}(1+\gamma_P + x,1)$ and $(X^Q_{n+1} \mid X_n^Q = x) \sim \mathrm{B}(1+\gamma_Q + x,1)$, s.t.\ the transition kernels $P$ and $Q$ have transition densities $p(x,\cdot) = \psi_{x + \gamma_P}(\cdot)$ and $q(x,\cdot) = \psi_{x + \gamma_Q }(\cdot)$, respectively.

Since for any $x \in [0,1]$ we have 
\[p(x,y) = (1 + x + \gamma_P) y^{x +\gamma_P} \geq (1+ \gamma_P) y^{1 + \gamma_P}, \quad x,y \in [0,1]\] 
and similarly $q(x,y) \geq  (1+ \gamma_Q) y^{1 + \gamma_Q}, x,y \in [0,1]$, it follows that for $\nu^P = \mathrm{B}(2 + \gamma_P,1), \nu^Q = \mathrm{B}(2 + \gamma_Q,1)$ we have 
\[P(x,\cdot) \geq \frac{1 + \gamma_P}{2 + \gamma_P} \nu^P(\cdot), \quad Q(x,\cdot) \geq \frac{1 + \gamma_Q}{2 + \gamma_Q} \nu^Q(\cdot), \quad x \in [0,1],
\] 
i.e., $X^Q$ and $X^P$ are uniformly ergodic. Moreover, for any $0< h \leq 1$ elementary calculations show  
\begin{align*}
h^{1 + \gamma_p - \gamma_q} Q(y,B(x,h)) \leq (2 + \gamma_Q) h^{1 + \gamma_p - \gamma_q} \int_{(x-h) \vee 0}^{(x+h) \wedge 1} y^{\gamma_Q} \diff{y} &\leq \frac{2 + \gamma_Q}{1 +\gamma_Q} \int_{(x-h) \vee 0}^{(x+h) \wedge 1} (2 + \gamma_P)y^{1 +\gamma_p} \diff{y}\\
&= \frac{2 + \gamma_Q}{1 +\gamma_Q} \nu^P(B(x,h)).
\end{align*}
It follows $(P,Q) \in \mathcal{T}(1+ \gamma_P - \gamma_Q, (1+ \gamma_Q)/(2+ \gamma_Q), 1)$. Noting that 
\[Q(0,B(0,h)) = h^{1 + \gamma_Q}, \quad \nu^P((0,h)) = h^{2 + \gamma_P}, \quad h \in (0,1],\]
we also see that $\gamma = 1+ \gamma_P - \gamma_Q$ is the minimal admissible kernel transfer exponent.
\end{example}

The statistically perhaps most intriguing situation concerns dimension reduction of the state space from source to target chain.

\begin{example}[$Q(x,\cdot)$ supported on lower-dimensional subspace] \label{ex:dim}
Suppose that $\cX = [0,1]^{\mathsf{d}}$ 
and that $\supp Q(x,\cdot) \subset \cX_Q \Subset \cX$ for all $x \in \cX$, where $\cX_Q$ is homeomorphic to $[0,1]^{\mathsf{d}_Q}$ for some $\mathsf{d}_Q < \mathsf{d}$. Let $\varphi$ be the associated homeomorphism and let $d_{\varphi}$ be the induced metric on $[0,1]^{\mathsf{d}_Q}$ given by $d_{\varphi}(x,y) = d(\varphi^{-1}(x),\varphi^{-1}(y))$, such that $(\cX_Q, d\vert_{\cX_Q})$ and $([0,1]^{\mathsf{d}_Q},d_{\varphi})$ are isometric. Assume that $\varphi$ is $\beta$-Hölder continuous for some $\beta \in (0,1]$ w.r.t.\ some vector norm induced metric $d^\prime$, i.e., on $[0,1]^{\mathsf{d}_Q}$ it holds 
\[d^\prime(\varphi(x),\varphi(y)) \leq L\lvert d(x,y)\rvert^\beta, \quad x,y \in \cX_Q,\] 
for some $L > 0$. Then, $\varphi(B_{d\vert_{\cX_Q}}(x,h)) \subset B_{d^\prime}(\varphi(x),Lh^\beta)$.
Since for any $x,y \in \cX_Q$ and $h \leq D$, we have 
\[Q(y,B_d(x,h)) = Q\big(y,B_{d\vert_{\cX_Q}}(x,h)\big) = Q\big(y,\varphi^{-1}(B_{d_\varphi}(\varphi(x),h)) \big),\]
it follows that if the image measure $Q(y,\cdot) \circ \varphi^{-1}$ has a Lebesgue density $q_{\varphi}(y,\cdot)$ on $[0,1]^{\mathsf{d}_Q}$ s.t.\ 
\[\sup_{y \in \cX_Q, z \in [0,1]^{\mathsf{d}_Q}} q_{\varphi}(y,z) \leq C < \infty,\] 
then for some constant $c(\mathsf{d}_Q,L,\beta)$ depending also on the choice of the metric $d^\prime$ we have
\begin{align*}
Q(y,B_d(x,h)) = \int_{B_{d_\varphi}(\varphi(x),h)} q_{\varphi}(y,z) \diff{z} \leq C\lebesgue_{\mathsf{d}_Q}(B_{d_\varphi}(\varphi(x),h)) &= C\lebesgue_{\mathsf{d}_Q}(\varphi(B_d(x,h)\cap \cX_Q))\\
&\leq C\lebesgue_{\mathsf{d}_Q}(B_{d^\prime}(\varphi(x),Lh^\beta))\\ 
&\leq Cc(\mathsf{d}_Q,L,\beta) h^{\beta \mathsf{d}_Q}.
\end{align*}
If we now assume that $P(x,\cdot)$ has a Lebesgue density $p(x,\cdot)$ on $\cX$ s.t.\ $\inf_{(x,y) \in \cX^2} p(x,y) \geq \varepsilon_P > 0$, then $P(x,\cdot) \geq \varepsilon_P \lebesgue\vert_{[0,1]^{\mathsf{d}}}$. Thus, $\nu^P \coloneq \lebesgue\vert_{[0,1]^{\mathsf{d}}}/\lebesgue([0,1]^{\mathsf{d}})$ is a minorizing measure for the kernel $P$ and we have $\nu^P(B(x,h)) \geq c^\prime h^{\mathsf{d}}$ for some constant $c^\prime > 0$ and any $x \in [0,1]^{\mathsf{d}}$. Thus, $P$ and $Q$ have transfer exponent $\gamma = \mathsf{d} - \beta \mathsf{d}_{Q}$ and $(P,Q) \in \mathcal{T}(\mathsf{d} - \beta \mathsf{d}_{Q},c^\prime/(Cc(\mathsf{d}_Q,L,\beta)))$. In particular, the transfer exponent $\gamma$ becomes larger with increasing dimension gap $\mathsf{d} - \mathsf{d}_Q$ and decreasing smoothness of the transfer map $\varphi$.
\end{example}

Having demonstrated the diversity of covariate shifts that can be captured by kernel transfer exponents, we now quantify its statistical relevance by proving that the existence of a transfer exponent for the Markov kernel pair $(P,Q)$ implies that the associated invariant distribution pair $(\pi^P,\pi^Q)$ belongs to an $\alpha$-family, where the parameter $\alpha$ is determined by the kernel transfer exponent and the  target dimension.

\begin{proposition} \label{prop:transfer}
Let $\mathcal{X} \Subset \R^{\mathsf{d}}$ with diameter $D > 0$. Suppose that $(P,Q) \in \mathcal{T}(\gamma,C,D)$ with $\cX_Q \Subset \cX$ s.t.\  $N_{\cX_Q}(\varepsilon, d\vert_{\cX_Q}) \leq (1+k_QD/\varepsilon)^{\mathsf{d}_Q}$ for some $\mathsf{d}_Q \leq \mathsf{d}$ and $k_Q \geq 1$. Then, 
\[(\pi^P,\pi^Q) \in \begin{cases} \mathcal{D}\Big(\gamma+\mathsf{d}_Q,\frac{(2k_Q+1)^{\mathsf{d}_Q}}{C\varepsilon_P}\Big), &\text{ if } \gamma + \mathsf{d}_Q \geq \mathsf{d},\\ \mathcal{D}^\prime\Big(\gamma+\mathsf{d}_Q, \mathsf{d}_Q, \frac{(2k_Q+1)^{\mathsf{d}_Q}}{(C\varepsilon_P)\wedge 1}\Big), &\text{ if } \gamma + \mathsf{d}_Q < \mathsf{d}.\end{cases}\]
In particular, we always have $(\pi^P ,\pi^Q) \in \mathcal{D}(\gamma + \mathsf{d},9^{\mathsf{d}}/(C\varepsilon_P))$.
\end{proposition}
\begin{proof} 
By the Doeblin recurrence condition and invariance of $\pi^P$ for $P$ we have 
\[\pi^P(A) = \int_{\cX} P^{m_P}(x,A) \, \pi^P(\diff{x}) \geq \varepsilon_P \nu^P(A) \pi^P(\cX) = \varepsilon_P \nu^P(A), \quad A \in \cB(\cX),\]
i.e., $\pi^P(\cdot) \geq \varepsilon_P \nu^P(\cdot)$. Moreover, if $\cX_Q \neq \cX$, since $\cX_Q$ is closed, for any $y \in \cX\setminus\cX_Q$ we can find $r > 0$ s.t.\ $B(y,r) \subset \cX \setminus \cX_Q$ and hence $Q^{m_Q}(x, B(y,r)) = 0$ for any $x \in \cX$. Hence, invariance of $\pi^Q$ for $Q$ yields $\pi^Q(B(y,r)) = 0$ and it follows that $\supp \pi^Q \subset \cX_Q$ as well. 
Hence, by our assumption $(P,Q) \in \mathcal{T}(\gamma,C,D)$ and invariance of $\pi^Q$ for $Q$, we have for any $h \leq D$ and $x \in \mathcal{X}_Q$,
\[\pi^Q(B(x,h)) = \int_{\cX_Q} Q^{m_Q}(y, B(x,h)) \, \pi^Q(\diff{y}) \leq C^{-1}(h/D)^{-\gamma} \nu^P(B(x,h)). \]
With the above it therefore follows 
\[C\varepsilon_P (h/D)^{\gamma}\pi^Q(B(x,h)) \leq \varepsilon_P \nu^P(B(x,h)) \leq \pi^P(B(x,h)), \quad x \in \cX_Q, h \leq D,\]
i.e., the distributions $(\pi^P,\pi^Q)$ have transfer exponent $\gamma$ with constant $C\varepsilon_P$ and radius $D$ in the sense of \cite{kpotufe21}. Since by assumption $N_{\cX_Q}(\varepsilon,d\vert_{\cX_Q}) \leq (1+ k_QD/\varepsilon)^{\mathsf{d}_Q}$, we can find $H \coloneq (1+ 2k_QD/h)^{\mathsf{d}_Q} $ points $x_1,\ldots,x_{H} \in \cX_Q$ s.t.\ $\cX_Q \subset \bigcup_{i=1}^{H} B(x_i,h/2)$ and by triangle inequality we have $B(x_i,h/2) \subset B(x,h)$ for any $x \in B(x_i,h/2)$. Hence, for any $h \leq D$,
\begin{align*} 
\rho_h(\pi^P,\pi^Q) &= \int_{\cX_Q} \frac{1}{\pi^P(B(x,h))} \, \pi^Q(\diff{x}) \\
&\leq \frac{D^\gamma}{C\varepsilon_P h^\gamma} \sum_{i=1}^{H} \int_{B(x_i,h/2)} \frac{1}{\pi^Q(B(x,h))} \, \pi^Q(\diff{x})\\\ 
&\leq \frac{D^\gamma}{C\varepsilon_P h^\gamma} \sum_{i=1}^{H} \int_{B(x_i,h/2)} \frac{1}{\pi^Q(B(x_i,h/2))} \, \pi^Q(\diff{x})\\ 
&= \frac{D^\gamma (1+2k_QD/h)^{\mathsf{d}_Q}}{ C\varepsilon_P h^\gamma}\\
&\leq  \frac{(2k_Q + 1)^{\mathsf{d}_Q}}{ C \varepsilon_P (h/D)^{\gamma + \mathsf{d}_Q}}.
\end{align*}
Upon noting that $\rho_h(\pi^Q,\pi^Q) \leq N_{\cX_Q}(h/2,d\vert_{\cX_Q})$, the first statement follows. The final assertion now is a simple consequence of the fact that by Lemma \ref{lem:cov} we always have 
\[N_{\cX_Q}(\varepsilon,d\vert_{\cX_Q}) \leq N^{\mathrm{ext}}_{\cX_Q}(\varepsilon/2,d\vert_{\cX_Q}) \leq N_{\cX}(\varepsilon/2,d) \leq  (1+4D/\varepsilon)^{\mathsf{d}}.\]
\end{proof}

Before we interpret this result, let us observe some concrete implications for a dimension gap  between training and target Markov chains as discussed in Example \ref{ex:dim}.
\begin{excont}[continues=ex:dim] 
Since $\varphi\colon \cX_Q \to [0,1]^{\mathsf{d}_Q}$ is an isometry, we have the covering number identity $N_{\cX_Q}(\varepsilon,d\vert_{\cX_Q}) = N_{[0,1]^{\mathsf{d}_Q}}(\varepsilon, d_{\varphi})$. 
If now also $\varphi^{-1}$ is $\beta^\prime$-Hölder continuous for some $\beta^\prime \in (0,1]$, i.e., $d(\varphi^{-1}(x),\varphi^{-1}(y)) \leq L^\prime \lvert d^\prime(x,y) \rvert^{\beta^\prime}$ for $x,y \in [0,1]^{\mathsf{d}_Q}$ with $L^\prime \geq 1$, then  $B_{d_{\varphi}}(x,\varepsilon) \supset B_{d^\prime}(x,(\varepsilon/L^\prime)^{1/\beta^\prime})$. Therefore, Lemma \ref{lem:cov} gives for $\varepsilon \leq 1$,
\[N_{\cX_Q}(\varepsilon,d\vert_{\cX_Q}) \leq \big(1 +2(L^\prime/\varepsilon)^{1/\beta^\prime} \big)^{\mathsf{d}_Q} \leq \big(1 +2L^\prime/\varepsilon \big)^{\mathsf{d}_Q/\beta^\prime}.\]
Since $(P,Q)$ has kernel transfer exponent $\gamma = \mathsf{d} - \beta\mathsf{d}_Q$, it now follows from Proposition \ref{prop:transfer} that the distribution pair $(\pi^P,\pi^Q)$ has an $\alpha$-family index 
\[\alpha = \begin{cases} \mathsf{d} + \tfrac{1-\beta\beta^\prime}{\beta^\prime}\mathsf{d}_Q, &\text{if } \mathsf{d}_Q/\beta^\prime \leq \mathsf{d},\\ 2\mathsf{d} - \beta\mathsf{d}_Q, &\text{if } \mathsf{d}_Q/\beta^\prime > \mathsf{d}.\end{cases}\] 
The convergence rates implied by Proposition \ref{prop:alpha} for $n_P \gg n_Q$ therefore coincide with the usual nonparametric rate for the source data dimension $\mathsf{d}$ if $\varphi$ and $\varphi^{-1}$ are Lipschitz and increase with decreasing Hölder smoothness of $\varphi$ and $\varphi^{-1}$.
\end{excont}

On a higher level, Proposition \ref{prop:transfer} demonstrates that without further information on the target kernel $Q$, the transfer exponent $\gamma$  cannot tell us reliably  by how much estimation rates are  slowed down due to the distributional shift in source and target data. Indeed, we can only guarantee an a priori $\alpha$-family index $\gamma + \mathsf{d} \geq \mathsf{d}$, which by Proposition \ref{prop:alpha} corresponds to the regime where the generalization risk w.r.t.\ the target distribution $\pi^Q$ is at least as large as the $\pi^P$-generalization risk without covariate shift. If for the effective target dimension in terms of the support covering number, $\mathsf{d}_Q$, we have $\mathsf{d}_Q < \mathsf{d}$, cf.\ Example \ref{ex:dim}, this $\alpha$-index may be significantly reduced and the model may even belong to the favorable regime $\alpha < \mathsf{d}$ depending on the interplay between $\gamma$ and $\mathsf{d}_Q$, cf.\ Example \ref{ex:alpha} in Appendix \ref{app:exalpha}. At the same time, it is important to note  that if the target dimension $\mathsf{d}_Q$ is strictly larger than the source dimension $\mathsf{d}$ associated to the Markov chain $X^P$, it follows from Lemma \ref{lem:explosion} that $\rho_h(\pi^P,\pi^Q)$ is infinite for $h$ small enough. Thus, for a covariate shift to a higher dimensional target, convergence guarantees for the generalization risk can only be achieved given a significant number of training observations generating the limiting distribution $\pi^Q$. This matches the intuition that in absence of further assumptions the regression function cannot be learned on completely unseen regions of the training data. 

Taking into account that the $\alpha$-index is proportional to the estimation rate, the concept of $\alpha$-families is therefore better suited for summarizing the difficulty of the covariate shift problem for Markov data compared to the kernel transfer exponent. The latter, however, remains a useful concept for verification of the $\alpha$-family property of source-target invariant distributions. In this sense, the findings from \cite{pathak22} in the scalar i.i.d.\ case are reinforced to multivariate Markov data.

\appendix
\section{Kernel pairs for $\alpha$-family index $\alpha < \mathsf{d}$} \label{app:exalpha}
By combining ideas from Example \ref{ex:beta} and Example \ref{ex:dim}, we construct  Markov kernels $(P,Q)$ s.t.\ their invariant distributions belong to an $\alpha$-family with $\alpha$-index strictly smaller than the dimension $\mathsf{d}$. In particular, for such Markov chains, the the NW estimator generalizes better to the target distribution than without covariate shift.
\begin{example}\label{ex:alpha}
Let $\mathsf{d} \geq 2$, $\cX = [0,1]^{\mathsf{d}}$ and $\cX_Q = [0,1]^{\mathsf{d}_Q} \times \{0\}^{\mathsf{d}-\mathsf{d}_Q}$ for some $\mathsf{d}_Q < \mathsf{d}$ and let $d = d_\infty$. Denote by $\mathbf{B}_n((\gamma_i)_{i=1,\ldots,n})$ the joint distribution of $n$ independent $\mathrm{B}(\gamma_i,1)$ distributed random variables for $\gamma_i > 0$, i.e., for $X \sim \mathbf{B}_n((\gamma_i)_{i=1,\ldots,n})$ we have 
\[\PP(X \in \diff{y})/\diff{y} =  \prod_{i=1}^n \gamma_i y_i^{\gamma_i-1} \one_{(0,1]}(y_i) \eqcolon \psi^n_{\bm{\gamma}}(y), \quad y \in [0,1]^{n}, \quad \bm{\gamma} = (\gamma_1,\ldots,\gamma_n).\] 
Let also $\iota_i$ be  functions s.t.\ $\iota_i\colon [0,1]^{\mathsf{d}} \to [\varepsilon,1]$ for some $\varepsilon \in (0,1]$ and $\bm{\gamma}^P \in (0,1)^{\mathsf{d}}, \bm{\gamma}^Q\in (0,1)^{\mathsf{d}_Q}$ s.t.\ $\gamma_i^P \geq \gamma_i^Q$ for any $i = 1,\ldots,\mathsf{d}_Q$. Let now $P$ be the Markov kernel on $[0,1]^{\mathsf{d}}$ defined by 
\[P(x,\diff{y}) = \psi^{\mathsf{d}}_{(\gamma^P_i \iota_i(x))_{i \in [\mathsf{d}]}}(y) \diff{y} \eqcolon p(x,y) \diff{y}, \quad y \in [0,1]^{\mathsf{d}},\] 
and $Q$ be the Markov kernel on $[0,1]^{\mathsf{d}}$ defined by 
\begin{align*} 
Q(x,\diff{y}) &= \psi^{\mathsf{d}_Q}_{(\gamma^Q_i \iota_i(x))_{i \in [\mathsf{d}_Q]}}(y^{\mathsf{d}_Q}) \diff{y^{\mathsf{d}_Q}} \prod_{i= \mathsf{d}_Q +1}^{\mathsf{d}} \delta_0(\diff{y_i})\\
&\eqcolon q(x,y^{\mathsf{d}_Q}) \diff{y^{\mathsf{d}_Q}} \prod_{i= \mathsf{d}_Q +1}^{\mathsf{d}} \delta_0(\diff{y_i}) , \qquad y^{\mathsf{d}_Q} \coloneq (y_1,\ldots,y_{\mathsf{d}_Q}), y \in [0,1]^{\mathsf{d}}.
\end{align*} 
In particular $\supp Q(x,\cdot) = \cX_Q$ for any $x \in \cX$. Since 
\[p(x,y) \geq \varepsilon^{\mathsf{d}} \psi^{\mathsf{d}}_{\bm{\gamma}^P}(y), \quad q(x,y^{\mathsf{d}_Q}) \geq \varepsilon^{\mathsf{d}_Q} \psi^{\mathsf{d}_Q}_{\bm{\gamma}^Q}(y^{\mathsf{d}_Q}), \quad x,y \in [0,1]^{\mathsf{d}},\] 
it follows that $P,Q$ generate uniformly ergodic Markov chains and $\nu^P \coloneq \bm{B}_{\mathsf{d}}(\bm{\gamma}^P)$ is a minorizing measure for $P$. Let $x \in \cX_Q$. Then, for any $y \in \cX_Q$ (see also Example \ref{ex:beta}),
\begin{align*} 
h^{\sum_{i=1}^{\mathsf{d}} \gamma_i^P - \sum_{i=1}^{\mathsf{d}_Q}\varepsilon\gamma_i^Q} Q(y, B(x,h)) &\leq \prod_{i= \mathsf{d}_Q+1}^\mathsf{d} h^{\gamma_i^P} \prod_{i=1}^{\mathsf{d}_Q} \gamma_i^Q h^{\gamma_i^P - \varepsilon\gamma_i^Q} \int_{(x_i-h) \vee 0}^{(x_i + h) \wedge 1}  z_i^{\varepsilon\gamma_i^Q -1} \diff{z_i}\\
&\leq \prod_{i= \mathsf{d}_Q+1}^\mathsf{d} h^{\gamma_i^P} \prod_{i=1}^{\mathsf{d}_Q} \frac{1}{\varepsilon} \int_{(x_i-h)\vee 0}^{(x_i + h)\wedge 1} \gamma_i^P z_i^{\gamma_i^P -1} \diff{z_i}\\
&= \varepsilon^{-{\mathsf{d}_Q}}\prod_{i= \mathsf{d}_Q+1}^\mathsf{d} \int_0^h \gamma_i^P z^{\gamma^P_i - 1} \diff{z_i}  \prod_{i=1}^{\mathsf{d}_Q}  \int_{(x_i-h)\vee 0}^{(x_i + h)\wedge 1} \gamma_i^P z_i^{\gamma_i^P -1} \diff{z_i}\\
&= \varepsilon^{-{\mathsf{d}_Q}}\nu^P(B(x,h)).
\end{align*}
It follows that 
\[(P,Q) \in \mathcal{T}\Big(\sum_{i=1}^{\mathsf{d}} \gamma_i^P - \varepsilon \sum_{i=1}^{\mathsf{d}_Q} \gamma_i^Q, \varepsilon^{\mathsf{d}_Q}, 1\Big),\]
and since $N_{\cX_Q}(\varepsilon, d_\infty\vert_{\cX_Q}) \leq (1+ 1/\varepsilon)^{\mathsf{d}_Q}$, it follows from Proposition \ref{prop:transfer} that 
\[(\pi^P,\pi^Q) \in \begin{cases} \mathcal{D}\Big(\sum_{i=1}^{\mathsf{d}} \gamma_i^P + \sum_{i=1}^{\mathsf{d}_Q} (1 - \varepsilon \gamma_i^Q),\frac{3^{\mathsf{d}_Q}}{\varepsilon^{\mathsf{d} + \mathsf{d}_Q}}\Big), &\text{ if } \sum_{i=1}^{\mathsf{d}} (1- \gamma_i^P) \leq \sum_{i=1}^{\mathsf{d}_Q} (1- \varepsilon \gamma_i^Q),\\ \mathcal{D}^\prime\Big(\sum_{i=1}^{\mathsf{d}} \gamma_i^P + \sum_{i=1}^{\mathsf{d}_Q} (1 - \varepsilon \gamma_i^Q), \mathsf{d}_Q, \frac{3^{\mathsf{d}_Q}}{\varepsilon^{\mathsf{d} + \mathsf{d}_Q}}\Big), &\text{ if } \sum_{i=1}^{\mathsf{d}} (1- \gamma_i^P) > \sum_{i=1}^{\mathsf{d}_Q} (1- \varepsilon\gamma_i^Q).\end{cases}\]
Hence, depending on the relation between the parameters $(\varepsilon,\bm{\gamma}^P,\bm{\gamma}^Q)$ and the dimensions $\mathsf{d},\mathsf{d}_Q$ of source and target Markov chains, $(\pi^P,\pi^Q)$ can belong to an $\alpha$-family with index $\alpha < \mathsf{d}$. In the simple scenario where $\varepsilon = \gamma_i^P = \gamma_j^Q$ for $i \in [\mathsf{d}], j \in [\mathsf{d}_Q]$, this is the case iff $\mathsf{d}_Q < \mathsf{d}/(1+\varepsilon)$, i.e., the target dimension must be sufficiently small to improve generalization rates compared to the situation without covariate shift.
\end{example}

\section{Auxiliary results}
\begin{lemma}\label{lem:explosion}
Let $(\mathcal{X},d)$ be a metric space and $\PP,\mathbb{Q}$ be probability measures on $(\cX,\cB(\cX))$. If for some set $\Lambda \in \cB(\cX)$ and $h > 0$ it holds $\QQ(\Lambda) > 0$ and $\PP(\Lambda_h) = 0$, where $\Lambda_h \coloneq \{x \in \mathcal{X} : d(x,\Lambda) \leq h\}$, then
$\rho_h(\PP,\QQ) = +\infty$.
\end{lemma} 
\begin{proof} 
Since for any $y \in \Lambda$ we have $B(y,h) \subset \Lambda_h$, it holds $\PP(B(y,h)) \leq \PP(\Lambda_h) = 0.$ Using that $\QQ(\Lambda) > 0$ we therefore obtain 
\[\rho_h(\PP,\QQ) \geq \int_{\Lambda} \frac{1}{\PP(B(y,h))} \, \QQ(\diff{y}) = +\infty.\]
\end{proof}

For completeness we recall the following standard covering  bound.
\begin{lemma}\label{lem:cov}
Let $d$ be some metric on $\R^{\mathsf{d}}$ and $\cX \subset \R^{\mathsf{d}}$ with $d$-diameter $0 < D < \infty$. Then, for any $\varepsilon > 0$, $N_{\cX}(\varepsilon,d\vert_{\cX}) \leq (1+2D/\varepsilon)^{\mathsf{d}}$.
\end{lemma}
\begin{proof} 
Let $B_{D/2} \subset \R^{\mathsf{d}}$ be some $d$-ball with radius $D/2$ s.t.\ $\cX \subset B_{D/2}$. Lemma 5.7 in \cite{wain19} gives $N_{B_{D/2}}(\varepsilon,d) \leq (1+D/\varepsilon)^{\mathsf{d}}$, hence the external covering number $N^{\mathrm{ext}}_{\cX}(\varepsilon,d)$ is bounded by $(1+D/\varepsilon)^{\mathsf{d}}$. Since $N_{\cX}(\varepsilon, d\vert_{\cX}) \leq N^{\mathrm{ext}}_{\cX}(\varepsilon/2,d)$ the result follows.
\end{proof}
\printbibliography
\end{document}